\documentclass[10pt,reqno]{amsart}
\usepackage{amssymb,mathrsfs,graphicx,extpfeil,amsmath,bbm}
\usepackage{ifthen,latexsym,float,colortbl}
\usepackage[margin=1in]{geometry}
\usepackage{caption}
\usepackage{dsfont}
\usepackage{arydshln}
\usepackage{sidecap}
\usepackage{rotating,dsfont,stackengine}
\usepackage{hyperref}
\usepackage{enumerate}
\DeclareMathOperator*{\esssup}{ess\,sup}

\usepackage{colortbl}
\definecolor{black}{rgb}{0.0, 0.0, 0.0}
\definecolor{red}{rgb}{1.0, 0.5, 0.5}
\provideboolean{shownotes} 
\setboolean{shownotes}{true} 
\newcommand{\margnote}[1]{
	\ifthenelse{\boolean{shownotes}}%
	{\marginpar{\raggedright\tiny\texttt{#1}}}%
	{}%
}
\newcommand{\hole}[1]{
	\ifthenelse{\boolean{shownotes}}%
	{\begin{center} \fbox{ \rule {.25cm}{0cm} \rule[-.1cm]{0cm}{.4cm}
				\parbox{.85\textwidth}{\begin{center} \texttt{#1}\end{center}} \rule
				{.25cm}{0cm}}\end{center}} {} }

\graphicspath{{pics/}}

\title[Global weak solutions to the nonlinear Vlasov--Fokker--Planck equation]
{Global existence of weak solutions to the nonlinear Vlasov--Fokker--Planck equation}

\author[Choi]{Young-Pil Choi}
\address[Young-Pil Choi]{\newline Department of Mathematics\newline
	Yonsei University, 50 Yonsei-Ro, Seodaemun-Gu, Seoul 03722, Republic of Korea}
\email{ypchoi@yonsei.ac.kr}

\author[Hwang]{Byung-Hoon Hwang}
\address[Byung-Hoon Hwang]{\newline Department of Mathematics Education\newline
	Sangmyung University, 20 Hongjimun 2-gil, Jongno-Gu, Seoul 03016, Republic of Korea}
\email{bhhwang@smu.ac.kr}

\author[Yoo]{Yeongseok Yoo}
\address[Yeongseok Yoo]{\newline Department of Mathematics\newline
	Yonsei University, 50 Yonsei-Ro, Seodaemun-Gu, Seoul 03722, Republic of Korea}
\email{mathysyoo@yonsei.ac.kr}

\numberwithin{equation}{section}

\newtheorem{theorem}{Theorem}[section]
\newtheorem{lemma}{Lemma}[section]

\newtheorem{proposition}{Proposition}[section]
\newtheorem{remark}{Remark}[section]
\newtheorem{definition}{Definition}[section]

\newcommand{\R}{\mathbb R}
\newcommand{\N}{\mathbb N}

\newcommand{\bbn}{\mathbb N}

\newcommand{\ls}{\lesssim}

\newcommand{\T}{\mathbb T}

\newcommand{\mc}{\mathcal C}

\newcommand{\bq}{\begin{equation}}
	\newcommand{\eq}{\end{equation}}
\newcommand{\e}{\epsilon}
\newcommand{\lt}{\left}
\newcommand{\rt}{\right}

\newcommand{\pa}{\partial}

\newcommand{\intr}{\int_{\R^N}}
\newcommand{\intrr}{\iint_{\T^N \times \R^N}}

\makeatletter
\def\moverlay{\mathpalette\mov@rlay}
\def\mov@rlay#1#2{\leavevmode\vtop{%
		\baselineskip\z@skip \lineskiplimit-\maxdimen
		\ialign{\hfil$\m@th#1##$\hfil\cr#2\crcr}}}
\newcommand{\charfusion}[3][\mathord]{
	#1{\ifx#1\mathop\vphantom{#2}\fi
		\mathpalette\mov@rlay{#2\cr#3}
	}
	\ifx#1\mathop\expandafter\displaylimits\fi}
\makeatother

\makeatletter
\newcommand*{\rom}[1]{\expandafter\@slowromancap\romannumeral #1@}
\makeatother

\begin{document}
	\allowdisplaybreaks
	
	\date{\today}
	
	\subjclass[]{}

	\begin{abstract} 
	In this paper, we study the nonlinear Fokker--Planck equation with fixed collision frequency. We establish the global-in-time existence of weak solutions to the equation with large initial data and show that our solution satisfies the conservation laws of mass, momentum, and energy, and Boltzmann's $H$-theorem.  
	\end{abstract}
	
	\maketitle \centerline{\date}

  \tableofcontents

	%
	%
	%
	%
	\setcounter{equation}{0}
	\section{Introduction}

	\subsection{Nonlinear Vlasov--Fokker--Planck equation}
The kinetic Fokker--Planck equation is one of the well-known variant models of the celebrated Boltzmann equation, describing the dynamics of particles in the presence of diffusion and friction. In \cite{Vi02}, a nonlinear form of the Fokker--Planck equation was presented where the equilibrium state is given by the local Maxwellian instead of the global one. With this modification, the equation fully satisfies the fundamental properties of the Boltzmann equation, the conservation laws of mass, momentum, and energy, and the $H$-theorem. On the other hand, the nonlinear Fokker--Planck equation can be regarded as a kinetic flocking model in the presence of noise, originating from the mean-field limit of the Motsch--Tadmor model \cite{MT11}, so it is sometimes called the Vlasov--Fokker--Planck (VFP) equation \cite{MV08} or the Fokker--Planck--Alignment model \cite{Shvpre}. In recent decades, VFP-type equations have been widely applied in various fields such as physics, astronomy, biology, economy, and social sciences \cite{DMPW09, F05, FPTT17, NPT10} for mathematical modeling. The aim of this paper is to study the Cauchy problem for the VFP equation with fixed collision frequency, i.e. the equation (20) of \cite{Vi02} with $\alpha=0$. To be more precise, let $f=f(x,v,t)$ be the distribution of particles on the phase space point $(x,v)\in \T^N\times\R^N$ with dimension $N\geq 1$ at time $t \in \R_+$. Then the evolution of $f$ is governed by the following kinetic equation:
	\begin{align}\label{eqn1.1}
		\partial_tf+v\cdot\nabla_xf =\mathcal{N}_{{\rm FP}}(f), 
	\end{align}
	with initial data
	\begin{align*}
		f(x,v,0)=: f_0(x,v).
	\end{align*}
	Here the first two terms describe the free transport of particles and the VFP operator $\mathcal{N}_{{\rm FP}}$ is given by
	\bq\label{q_op}
	\mathcal{N}_{{\rm FP}}(f) := \nabla_v\cdot\lt(T_f\nabla_vf+(v-u_f)f\rt),
	\eq
	where the collision frequency is set to be unity. The observable quantities $\rho_f = \rho_f(x,t)$, $ u_f =  u_f(x,t)$, and $T_f = T_f(x,t)$ represent the mass density, mean velocity, and temperature, defined by the following relations:
	\begin{align*}
		\rho_f=\int_{\mathbb{R}^N}f\, dv,\quad
		\rho_f u_f=\int_{\mathbb{R}^N}vf\, dv,\quad \mbox{and} \quad
		N\rho_f T_f=\int_{\mathbb{R}^N}|v-u_f|^2f\, dv,
	\end{align*}
	respectively.  Note that it is natural to consider $u_{f}$ and $T_{f}$ as
	\begin{equation*}
		u_{f}(x,t)
		=\begin{cases}
			\frac{\rho_f u_f(x,t)}{\rho_f(x,t)} & \text{if $\rho_f(x,t) \neq 0$}\cr
			0 & \text{if $\rho_f(x,t) = 0$}
		\end{cases}
		\quad \mbox{and} \quad 
		T_{f}(x,t)
		=\begin{cases}
			\frac{\rho_f T_f(x,t)}{\rho_f(x,t)} & \text{if $\rho_f(x,t) \neq 0$}\cr
			0 & \text{if $\rho_f(x,t) = 0$}
		\end{cases},
	\end{equation*}
	respectively.  
	

 The study of the existence theory for the VFP equation is by now a well-established research topic. Some of the previous works on the existence theory and large-time behavior of solutions for the VFP equation can be summarized as follows. Since it has been studied by many authors, we do not intend to exhaust references in this paper. For the linear VFP equation, which is the equation \eqref{eqn1.1} with $T_f \equiv 1$ and $u_f \equiv 0$, the Vlasov--Fokker--Planck system is considered in \cite{Deg86} where rigorous treatments regarding the existence, uniqueness, and regularity are provided, and its Green function is addressed in  \cite{Bou93}. The convergence rate of solutions for the linear VFP equation with the external potential to equilibrium is also discussed in \cite{DV01, Her07, HN04}. We would like to remark that the operator \eqref{q_op} is also known as the hypo-coercive operator \cite{Vi09}, which is involved in the rate of convergence to equilibrium.  Concerning the VFP equation, when the temperature is given as a fixed positive constant, i.e. \eqref{eqn1.1} with $T_f \equiv 1$, the global classical solutions near the global Maxwellian and its large-time behavior are studied in \cite{C16}. Note that in that case, the equation satisfies the conservation of mass and momentum, but not the energy.  In \cite{LY21}, the global existence of regular solutions to the equation \eqref{eqn1.1} with non-constant collision frequency near the global Maxwellian and its asymptotic stability are provided. In \cite{MM16}, the ellipsoidal VFP equation is proposed to obtain the correct Prandtl number in the compressible Navier--Stokes asymptotics. Roughly speaking, the diffusion coefficient is replaced by a non-diagonal temperature tensor, but the model equation still satisfies the conservation laws and $H$-theorem. Later, the global-in-time existence and uniqueness of regular solutions to that ellipsoidal VFP equation near the global Maxwellian and its large-time behavior are obtained in \cite{SJ19}. We also refer to recent works \cite{AC23, FN22} for the developments of the numerical scheme for the simulations of the VFP operator and the study on the multi-species VFP kinetic model conserving mass, momentum, and energy as well as satisfying Boltzmann's $H$-theorem.

Despite those significant developments in the study of the global-in-time existence of regular solutions near the global Maxwellian, the literature on the existence of solutions to the VFP equations with large initial data remains extremely limited. In \cite{KMT13}, the existence of weak solutions for VFP equations with constant temperature was studied, and to the best of our knowledge, no other results have been reported so far. In the present work, for the first time, we obtain the global existence of weak solutions to \eqref{eqn1.1} with large initial data.

	%
	%
	%
	%
\subsection{Conservation laws and entropy inequality} Here, we present formal computations showing that the equation \eqref{eqn1.1} satisfies the conservation of mass, momentum, and energy, and Boltzmann's $H$-theorem. First, it is clear that
\begin{equation}\label{cancellation}
\intr (1, v, |v|^2) \mathcal{N}_{\rm FP}(f)\,dv = 0	
\end{equation}
due to
\[
\intr (v - u_f) f \,dv = 0,\quad \mbox{and}\quad \intr |v|^2f\,dv=N\rho_fT_f+\rho_f|u_f|^2.
\]
This gives the conservation laws of \eqref{eqn1.1}:
\[
\frac{d}{dt}\intrr (1, v, |v|^2) f\,dvdx = 0.	
\]
Moreover, if we define the local Maxwellian $M_f$ by
\[
M_f = \frac{\rho_f}{(2\pi T_f)^\frac{N}{2}} \exp\lt( - \frac{|v - u_f|^2}{2 T_f} \rt),
\]
then we find that the nonlinear Fokker--Planck operator $\mathcal{N}_{\rm FP}(f)$ can be written as
\[
\mathcal{N}_{\rm FP}(f) = \nabla_v \cdot \lt(T_f f \nabla_v \log \lt( \frac{f}{M_f}\rt) \rt).
\]
This together with \eqref{cancellation} gives
\[
\intr \mathcal{N}_{\rm FP}(f) \log M_f \,dv = 0,
\]
which yields
\begin{align}\label{dss_ent}
\begin{aligned}
\intr  \mathcal{N}_{\rm FP}(f) \log f\,dv &= \intr  \mathcal{N}_{\rm FP}(f) \log \lt( \frac{f}{M_f}\rt)\,dv \cr
&= - T_f \intr f \lt| \nabla_v \log \lt( \frac{f}{M_f}\rt)\rt|^2 dv \cr
&= -\frac1{T_f}\intr \frac1f |T_f \nabla_v f - (u_f -v)f|^2\,dv\cr
&\leq 0.
\end{aligned}
\end{align}
Thus, multiplying \eqref{eqn1.1} by $1+\ln f$ and integrating over $(x,v)$, we also have the $H$-theorem:
\[
\frac{d}{dt}\intrr f \log f\,dvdx = \intrr  \mathcal{N}_{\rm FP}(f) \log f\, dvdx \leq 0.
\]
	
	%
	%
	%
	%

	\subsection{Notations and main result}
For functions $f(x,v)$ and $g(x)$, when there is no risk of confusion, we will use  $\left\|f\right\|_{L^p}$ and $\left\|g\right\|_{L^p}$ to denote the usual $L^p(\mathbb{T}^N\times \mathbb{R}^N)$-norm and $L^p(\mathbb{T}^N)$-norm, respectively. We also use a weighted $L^p$-norm defined by
	\[
	\|f\|_{L_{q}^p}:=\left( \iint_{\mathbb{T}^N\times\mathbb{R}^N} \big(1+|v|^q\big)|f(x,v)|^p\, dxdv\right)^{\frac{1}{p}}\quad \mbox{if}\quad p\in [1,\infty)
	\]
 and 
	\[
	\|f\|_{L_q^\infty}:=\esssup_{(x,v)\in \mathbb{T}^N\times\mathbb{R}^N} \big(1+|v|^q\big)|f(x,v)|.
	\]
	Note that $\|f\|_{L_{q}^p}$ does not converge to $\|f\|_{L_q^\infty}$ as $p \to \infty$. 
	We denote a generic but not identical positive constant by $C$, and $C_{\alpha, \beta, \dots}$ or $C(\alpha, \beta, \dots )$ stands for a  positive constant depending on $\alpha$, $\beta$, and etc.
	
	Now, we introduce our notion of weak solutions to the equation \eqref{eqn1.1} and state our main theorem.
	\begin{definition}
		\label{def1.1}
		We define a weak solution $f$ to the equation \eqref{eqn1.1} if the following conditions are satisfied: 
		\begin{enumerate}[(i)]
			\item $f\in L^{\infty}(0,T;L^1\cap L^{\infty}(\mathbb{T}^N\times\mathbb{R}^N))$
			\item for all $\phi(x,v,t) \in \mc_c^2(\mathbb{T}^N \times \R^N \times[0,T])$ which has a compact support in $v$ 
			\mbox{with} $\phi(x,v,T)=0$,
			\begin{align*}
				&-\iint_{\mathbb{T}^N\times\mathbb{R}^N}f_0\phi_0\, dvdx -\int_{0}^{T}\iint_{\mathbb{T}^N\times\mathbb{R}^N}f\left( \partial_t\phi +v\cdot\nabla_x\phi+(u_f-v)\cdot\nabla_v\phi\right) \, dvdxdt\\
				& \quad =\int_{0}^{T}\iint_{\mathbb{T}^N\times\mathbb{R}^N} T_ff \Delta_v\phi\, dvdxdt.
			\end{align*}
		\end{enumerate}
	\end{definition}

	\begin{theorem}\label{main result}
		Let $T>0$. Suppose that the initial data $f_0\ge 0$ satisfies
\bq\label{ini_condi}
f_0 \in L^1_3\cap L^{\infty}(\mathbb{T}^N\times\mathbb{R}^N) \quad \mbox{and} \quad f_0 \log f_0 \in L^1(\T^N \times \R^N).
\eq
	 Then the equation \eqref{eqn1.1} admits at least one weak solution $f(x,v,t)\ge 0$  in the sense of Definition \ref{def1.1} such that
		\begin{enumerate} [(i)]
			\item  $\displaystyle
				\sup_{0\leq t\leq T}\|f(\cdot,\cdot,t)\|_{L^1 \cap L^{\infty} }\leq C_T \|f_0\|_{L^1 \cap L^{\infty}}$ and		 
			\item the third velocity-moment on $f$ is uniformly bounded:
			\begin{align*}
				\sup_{0\leq t\leq T}\iint_{{\mathbb{T}^N}\times\mathbb{R}^N} |v|^3 f(x,v,t)\, dvdx<\infty.
			\end{align*}
		\end{enumerate}
		Moreover, we have the following conservation laws and Boltzmann's $H$-theorem:
		\bq\label{conser_law}
		\intrr (1, v, |v|^2) f\,dvdx = \intrr (1, v, |v|^2) f_0\,dvdx
		\eq
		and
		\bq\label{h-theorem}
		\intrr f \log f\,dvdx \leq \intrr f_0 \log f_0\,dvdx.
		\eq
	\end{theorem}

	\begin{remark} By using Young's inequality, $f\in L^1_3(\T^N \times \R^N)$ implies $f\in L^1_\alpha(\T^N \times \R^N)$ for any $\alpha \in [0,3]$.
	\end{remark}

 \begin{remark}
     The second condition in \eqref{ini_condi} is only required for the $H$-theorem \eqref{h-theorem}.
 \end{remark}
	%
	%
	%
	%
	\subsection{Strategy of the proof} The proof of Theorem \ref{main result} is mainly divided into three main steps. 
	
	We first introduce a regularized equation associated with \eqref{eqn1.1}. To remove any singularities in the mean velocity and temperature, motivated from \cite{CY20}, we regularize them by introducing two regularization parameters $\e$ and $\delta$. We emphasize that there are various ways to regularize the equation, but it should be chosen appropriately so that the solutions recovered from the regularization satisfy the desired conservation laws \eqref{conser_law} and entropy inequality \eqref{h-theorem}. For instance, in our work, the temperature $T_f$ is regularized in this way:
		$$
		T_f^{\e, \delta}(x,t):=
		\frac{\Phi_f^{\e, \delta} + \delta^2}{N\rho_f*\theta_\epsilon+ \delta\big( 1+\Phi_f^{\e, \delta} \big) },
		$$
		with
		$$
		\Phi_f^{\e, \delta}
		:=\left( N\rho_fT_f+\rho_f|u_f|^2\right) * \theta_\epsilon - \frac{\left| \rho_fu_f*\theta_\epsilon\right|^2}{\rho_f*\theta_\epsilon+\delta\big( 1+\left| \rho_fu_f*\theta_\epsilon\right|^2\big)}.
		$$
Here $\Phi_f$ is designed to take advantage of the strong convergence of the moments of the regularized solutions, see \eqref{novel2} below for more details.

 After regularizing the equation, we do not have any singular terms in the equation. However, it still has strongly nonlinear terms, and thus the existence of solutions to that regularized equation is not clear. For that reason, in the second step, we develop an existence theory for the regularized equation. For this, we construct approximate solutions to that regularized equation. By obtaining appropriate uniform bound estimates of the sequence of approximate solutions, we show that the sequence is Cauchy in $L^\infty(0, T;L^2_q(\T^N \times \R^N))$ for some positive constant $q$ large enough. Thus, we have a limit function $f_{\e,\delta}$ from the approximate solutions, and we provide that the limit function is the weak solution to that regularized equation.

 Finally, we send the regularization parameters $\e,\delta$ to zero to recover our main equation \eqref{eqn1.1}. Moreover, we show that our constructed weak solutions satisfy the conservation laws and the $H$-theorem. This final step involves the weak and strong compactness via the velocity-moment estimate and velocity averaging lemma. We would like to remark that double regularization is not required for the existence of solutions. In our strategy, we need it to show that our constructed weak solutions satisfy the $H$-theorem. Precisely, we first send $\e \to 0$ and then $\delta \to 0$. Clearly, the temperature $T_f$ would be zero in some region of $\T^N \times (0, T)$ with positive measure, and it causes some technical difficulties in analyzing the entropy inequality, see \eqref{dss_ent} for instance. To overcome that difficulty, we regularize the temperature so that even after sending $\e \to 0$, it is still bounded from below by some positive constant $C(\delta)$. We then bound the regularized temperature, which is now independent of $\e$, from below and above appropriately to pass to the limit $\delta \to 0$.
 
	%
	%
	%
	%
	\subsection{Organization of the paper}
	
	The rest of this paper is organized as follows. In Section \ref{sec:reg}, we introduce the regularized equation associated to our main equation \eqref{eqn1.1}. We also introduce a sequence of solutions approximating the regularized equation. We then provide some uniform bounds of the approximate solutions. Section \ref{sec:reg_weak} is devoted to showing the existence of weak solutions to the regularized equation. We also present some uniform bound estimates of the solutions with respect to the regularization parameters. Finally, in Section \ref{sec:main}, we provide the details of the proof for Theorem \ref{main result}.
	%
	%
	%
	%
	%
	%
	
	\section{A regularized equation}\label{sec:reg}
	We need to regularize the equation \eqref{eqn1.1} to remove the singularity in the local alignment force. Let us use a mollifier $\theta_\epsilon(x)=\epsilon^{-N}\theta({\frac{x}{\epsilon}})$ with $\e \in (0,1)$ to regularize the mean velocity $u_f$ and the temperature $T_f$, which satisfies
	\begin{align*}
		\int_{\mathbb{T}^N}\theta(x)\, dx=1.
	\end{align*}
	\\
	Then the regularized equation of \eqref{eqn1.1} is defined as
	\begin{align}
		\label{eqn2.1}
		\partial_tf_{\e,\delta}+v\cdot\nabla_xf_{\e,\delta}  =\nabla_v\cdot\lt(T_{f_{\e,\delta}}^{\e, \delta}\nabla_vf_{\e,\delta}+(v-u_{f_{\e,\delta}}^\epsilon)f_{\e,\delta}\rt)
	\end{align}
	with regularized initial data
\bq\label{reg_ini}
		f_{\e,\delta}(x,v,0)=: f_{0,\epsilon}(x,v).
\eq
	Here the regularized mean velocity $u_f^\epsilon$ and temperature $T_f^{\e, \delta}$ are given by
	\begin{align*}
		u_f^\epsilon(x,t):=\frac{\left(\rho_f u_f\right)*\theta_\epsilon}{\rho_f*\theta_\epsilon+\epsilon\big(1+|(\rho_f u_f)*\theta_\epsilon|^2\big)} \quad \mbox{and} \quad 	T_f^{\e, \delta}(x,t):=
		\frac{\Phi_f^{\e, \delta} + \delta^2}{N\rho_f*\theta_\epsilon+ \delta\big( 1+\Phi_f^{\e, \delta} \big) },
	\end{align*}
	with
	\begin{align*}
		\Phi_f^{\e, \delta}
		:=\left( N\rho_fT_f+\rho_f|u_f|^2\right) * \theta_\epsilon - \frac{\left| \rho_fu_f*\theta_\epsilon\right|^2}{\rho_f*\theta_\epsilon+\delta\big( 1+\left| \rho_fu_f*\theta_\epsilon\right|^2\big)  },
	\end{align*}
respectively, where $\delta \in (0,1)$,
	\begin{align*}
		\rho_f=\int_{\mathbb{R}^N}f\, dv,\quad
		\rho_f u_f=\int_{\mathbb{R}^N}vf\, dv,\quad
		\mbox{and}\quad
		N\rho_f T_f=\int_{\mathbb{R}^N}|v-u_f|^2f\, dv.
	\end{align*}
	The regularized initial data $f_{0,\e}$ is given by
\bq\label{f_0ep}
		f_{0,\epsilon}=\left\lbrace \eta_\epsilon* f_0\right\rbrace (x,v)+\epsilon{e^{-|v|^2}}
\eq
for all $\delta \in (0,1)$, where $\eta$ is a standard mollifier and $\eta_\epsilon(x,v)=\epsilon^{-2N}\eta({\frac{x}{\epsilon}},{\frac{v}{\epsilon}})$.
	
	\begin{remark}\label{rem_poT}Note that if $f \geq 0$, then
	\[
	\left| \rho_fu_f*\theta_\epsilon\right|^2 \leq (\rho_f*\theta_\epsilon) |( \rho_f|u_f|^2)  * \theta_\epsilon|
	\]
thanks to H\"older's inequality, and thus 
	\[
	\Phi_f^{\e, \delta} \geq N \rho_f T_f * \theta_\epsilon \geq 0.
	\]
	This, in particular, implies $T_f^{\e, \delta} \geq 0$.
	\end{remark}

For the regularized equation \eqref{eqn2.1}, we obtain the following proposition on the global-in-time existence of weak solutions.
	\begin{proposition}
		\label{prop2.1}
		Let $T>0$.  Suppose that the initial data $f_0$ satisfies \eqref{ini_condi} and $f_{0,\e}$ is given by \eqref{f_0ep}. Then there exists a weak solution $f_{\epsilon,\delta}$ of the regularized equation \eqref{eqn2.1}--\eqref{reg_ini} on the interval $[0, T]$ in the sense of Definition \ref{def1.1}. Furthermore, there exists $C>0$ independent of $\epsilon$ and $\delta$ such that 
				\begin{align*}
			\sup_{0\leq t\leq T}\|f_{\e,\delta}(t)\|_{L^1 \cap L^\infty}\leq \|f_{0,\e}\|_{L^1 \cap L^\infty} e^{NT}
		\end{align*}
		and
		\begin{align*}
			\sup_{0\leq t\leq T}\iint_{\mathbb{T}^N\times\mathbb{R}^N}|v|^3 f_{\e,\delta}(x,v,t)\, dvdx
			\leq C.
		\end{align*}
	\end{proposition}
	
\begin{remark}Due to the regularized terms $T^{\e,\delta}_f$, $u^\e_f$, and the regularized initial data, we can also have the global existence of regular solutions $f_{\e,\delta}$ to the equation \eqref{eqn2.1} for fixed $\e,\delta > 0$. However, we can only obtain the uniform bound estimate of solutions $f_{\e,\delta}$ in $L^1\cap L^\infty$ space. 
\end{remark}	
	
	In order to prove Proposition \ref{prop2.1}, we introduce another equation, approximating the regularized equation \eqref{eqn2.1}, in the next subsection.
	%
	%
	%
	%
	%
	%
	\subsection{A regularized and linearized equation}
	To show the existence of weak solutions to the regularized equation \eqref{eqn2.1}, we construct approximate solutions by solving the following equation:
	\begin{align}
		\label{eqn2.2}
		\partial_tf^{n+1}_{\e,\delta}+v\cdot\nabla_xf^{n+1}_{\e,\delta}  =\nabla_v\cdot\lt(T_{f^n_{\e,\delta}}^{\epsilon, \delta}\nabla_vf^{n+1}_{\e,\delta}+\lt(v-u_{f^n_{\e,\delta}}^{\epsilon}\rt)f^{n+1}_{\e,\delta}\rt)
	\end{align}
	with the initial data and first iteration step:
\[
f^n_{\e,\delta}(x,v,t)|_{t=0}=f_{0,\e}(x,v) \quad \mbox{for all} \quad n \geq 1, \quad (x,v) \in \T^N \times \R^N
\]
and
\[
f_{\e,\delta}^0(x,v,t)=f_{0,\e}(x,v) \quad (x,v,t) \in \T^N \times \R^N \times \R_+.
\]

Our main goal of this subsection is to provide the global-in-time existence and uniqueness of the solution of the regularized and linearized equation \eqref{eqn2.2} and uniform-in-$n$ estimates of the solution. For notational simplicity, we shall drop the subscript $\e$ and $\delta$, and denote $f^n_{\e,\delta}$ by $f^n$.
	\begin{proposition}
		\label{prop2.2}
		Let $T>0$. For any $n\in \mathbb{N}$, the regularized and linearized equation \eqref{eqn2.2} admits  a unique solution $f^n\in L^{\infty}(0,T;L_q^{\infty}(\mathbb{T}^N\times\mathbb{R}^N))$ satisfying
		\begin{itemize}
		 \item[(i)]  The distribution $f^n$ is positive:
		$$
f^n(x,v,t)>0 \quad \mbox{for all } n\in \mathbb{N}.	
$$	 
		\item[(ii)]   There exist positive constants $C_T$ and $C_{\e,\delta,T}$ such that 
				\[
		\sup_{0 \leq t \leq T}\sup_{n \in \N}\|f^n(t)\|_{L^1 \cap L^\infty} \leq C_T\|f_{0,\e}\|_{L^1 \cap L^\infty}
		\]
		and
				\[
		\sup_{0 \leq t \leq T}\sup_{n \in \N}\|f^n(t)\|_{L^\infty_q} \leq C_{\e,\delta,T}\|f_{0,\e}\|_{L^\infty_q}.
		\]
		\item[(iii)]   There exists $c_{\e,\delta} > 0$  such that 
		\[
			\inf_{(x,t) \in \T^N \times (0,T)} \inf_{n \in \N}\rho_{f^n}(x,t) \geq c_{\e,\delta}.
		\]
		\end{itemize}
	\end{proposition}

Before proving the above proposition, we list some useful bound estimates in the lemma below for later use.
	\begin{lemma}Let $f\ge 0$ and $\e,\delta \in (0,1)$. The following relations hold.
		\label{lemma2.1}
		\begin{itemize}
		\item[(i)] Let $q>N+2$. Then  we have
		\begin{flalign*}
		&\|\rho_{f}\|_{L^\infty}
		+\|\rho_{f}u_{f}\|_{L^\infty}
		+\|\rho_{f}T_{f}\|_{L^\infty}
		\leq C\|f\|_{L_q^\infty},
		\end{flalign*}
		where $C>0$ depends only on $q$ and $N$.
		\item[(ii)] For $u_{f}^{\epsilon}, \Phi_f^{\epsilon,\delta}$ and $T_{f}^{\epsilon, \delta}$, we have
		\begin{flalign*}
			\| u_{f}^{\epsilon}\|_{L^\infty} \leq \frac1\e,  \qquad 
			\|T_{f}^{\epsilon, \delta}\|_{L^\infty}
			\leq \frac1\delta,
		\end{flalign*}
and
		\begin{flalign*}
			\|\nabla_xu_f^{\epsilon}\|_{L^{\infty}}+\|\nabla_x\Phi_f^{\epsilon,\delta}\|_{L^{\infty}}
			+\|\nabla_xT_f^{\epsilon,\delta}\|_{L^{\infty}}
			\leq C_{\epsilon,\delta}\|f\|_{L_q^\infty}.
		\end{flalign*}
			\item[(iii)]There exists a positive constant $C$, depending only on $N$, such that 
		\begin{align*}
			\rho_{f}\leq C\|f\|_{L^{\infty}}({T_{f}})^{\frac{N}{2}}.
		\end{align*}
		\end{itemize}
	\end{lemma}
	\begin{proof}
(i) Since $q>N+2$, we easily obtain 
		\begin{align*}
		&|\rho_{f}|
		\leq\int_{\mathbb{R}^N}(1+|v|^q)^{-1}(1+|v|^q)f\,dv
		\leq C\|f\|_{L_q^\infty},
		\end{align*}
		\begin{align*}
		&|\rho_{f} u_{f}|
		\leq\int_{\mathbb{R}^N}|v|(1+|v|^q)^{-1}(1+|v|^q)f\,dv
		\leq C\|f\|_{L_q^\infty},
		\end{align*}
		and
		\begin{align*}
		|\rho_{f} T_{f}|
		= {\frac{1}{N}}\int_{\mathbb{R}^N}|v|^2f\,dv+{\frac{1}{N}}\rho|u_f|^2
		\le  {\frac{2}{N}}\int_{\mathbb{R}^N}|v|^2(1+|v|^q)^{-1}(1+|v|^q)f\,dv
		\leq C\|f\|_{L_q^\infty},
		\end{align*}
		where we used H\"{o}lder's inequality so that
		\begin{align*}
		\rho_{f}|u_{f}|^2
		=\left| \int_{\mathbb{R}^N}vf\,dv\right| ^2\left( \int_{\mathbb{R}^N}f\,dv\right) ^{-1}
		\leq\int_{\mathbb{R}^N}|v|^2f\,dv.
		\end{align*}
(ii) By the definition of $u_{f}^{\epsilon}$ and $T_{f}^{\epsilon, \delta}$, we get
		\begin{align*}
			|u_{f}^{\epsilon}|
			=&\left| {(\rho_{f} u_{f})*\theta_\epsilon\over (\rho_{f} * \theta_\epsilon)+\epsilon(1+|(\rho_{f} u_{f})*\theta_\epsilon|^2)}\right| 
			\leq \left| {1+|(\rho_{f} u_{f})*\theta_\epsilon|^2\over \epsilon(1+|(\rho_{f} u_{f})*\theta_\epsilon|^2)}\right| 
			\leq {\frac{1}{\epsilon}},
		\end{align*}
		and
		\begin{align*}	
			 T_{f}^{\epsilon, \delta}
			=&  \frac{\Phi_f^{\e, \delta} + \delta^2}{N\rho_{f}*\theta_\epsilon+\delta\big( 1+\Phi_f^{\e, \delta} \big) } 
			\leq  \frac{1+\Phi_f^{\e, \delta}}{\delta\big( 1+\Phi_f^{\e, \delta} \big) } 
			\leq \frac1\delta.
		\end{align*}
Next, we use the previous result (i) to obtain
		\begin{align*}
			|\nabla_xu_f^{\epsilon}|
			\leq& \left|{\frac{(\rho_f u_f)*\nabla_x\theta_\epsilon}{\rho_f*\theta_\epsilon+\epsilon(1+|(\rho_f u_f)*\theta_\epsilon|)}}\right|
			\cr&+\left|{\frac{(\rho_f u_f)*\theta_\epsilon(|\rho_f*\nabla_x\theta_\epsilon|+2\epsilon|\rho_fu_f*\theta_\epsilon||\rho_fu_f*\nabla_x\theta_\epsilon|)}{ (\rho_f*\theta_\epsilon+\epsilon(1+|(\rho_f u_f)*\theta_\epsilon|))^2}}\right|
			\cr\leq&  \frac{1}{\e}\left|(\rho_f u_f)*\nabla_x\theta_\epsilon\right|
			+\frac{1}{\e^2}\left|\rho_f*\nabla_x\theta_\epsilon\right|
			+\frac{2}{\e}\left|\rho_fu_f*\nabla_x\theta_\epsilon\right|
			\cr\leq& C_\epsilon\|f\|_{L_q^{\infty}}
		\end{align*}
		and
		\begin{align*}
			\left|\nabla\Phi_f^{\epsilon,\delta}\right|
			\leq& \left|\left( N\rho_{f} T_f+\rho_f\left|u_f\right|^2\right) *\nabla\theta_\epsilon\right|
			+\frac{2\left|\rho_f u_f*\nabla\theta_\epsilon\right|\left|\rho_f u_f*\theta_\epsilon\right|}{\rho_f*\theta_\epsilon+\delta\left( 1+\left| \rho_fu_f*\theta_\epsilon\right|^2\right) }\cr
			&+\frac{\left|\rho_fu_f*\theta_\epsilon\right|^2 \left|\rho_f*\nabla\theta_\epsilon\right|}{\left( \rho_f*\theta_\epsilon+\delta\left( 1+\left| \rho_fu_f*\theta_\epsilon\right|^2\right)\right)^2 }
			+\frac{2\left|\rho_fu_f*\nabla\theta_\epsilon\right|\left|\rho_fu_f*\theta_\epsilon\right|^3}{\left( \rho_f*\theta_\epsilon+\delta\left( 1+\left| \rho_fu_f*\theta_\epsilon\right|^2\right)\right)^2 }\cr
			\leq& \left|\left( \int_{{\mathbb{R}^N}}|v|^2f~dv\right) *\nabla\theta_\epsilon\right|
			+\frac{2}{\delta}\left|\rho_fu_f*\nabla\theta_\epsilon\right|
			+\frac{1}{\delta^2}\left|\rho_f*\nabla\theta_\epsilon\right|
			+\frac{2}{\delta^2}\left|\rho_fu_f*\nabla\theta_\epsilon\right|\cr
			\leq& C_{\epsilon,\delta}\|f\|_{L_q^\infty}.
		\end{align*}
Using the above result, we further estimate
		\begin{flalign*}
			|\nabla_xT_f^{\epsilon,\delta}|
			\leq&\frac{\left|\nabla\Phi_f^{\epsilon,\delta}\right|}{N\rho_{\epsilon}*\theta_\epsilon+\delta\left( 1+\Phi_f^{\epsilon,\delta}\right)}
			+\frac{\left(\left|\Phi_f^{\epsilon,\delta}\right|+1\right) \left|N\rho_f*\nabla\theta_\epsilon\right|}{\left( N\rho_{\epsilon}*\theta_\epsilon+\delta\left( 1+\Phi_f^{\epsilon,\delta}\right)\right) ^2}
			+\frac{\left(\left|\Phi_f^{\epsilon,\delta}\right|+1\right) \left|\delta\nabla\Phi_f^{\epsilon,\delta}\right| }{\left( N\rho_{\epsilon}*\theta_\epsilon+\delta\left( 1+\Phi_f^{\epsilon,\delta}\right)\right) ^2}\cr
			\leq& \frac{1}{\delta}\left|\nabla\Phi_f^{\epsilon,\delta}\right|
			+\frac{N}{\delta^2}\left|\rho_f*\nabla\theta_\epsilon\right|
			+\frac{1}{\delta}\left|\nabla\Phi_f^{\epsilon,\delta}\right|\cr
			\leq& C_{\epsilon,\delta}\|f\|_{L_q^\infty}.
		\end{flalign*}
(iii) Even though this proof is available in \cite{PP93}, for completeness, we give details of it here. For $\rho_{f}$, we have
	\begin{align*}
		\rho_{f}
		&\leq {\frac{1}{R^2}}\int_{|v-u|>R}|v-u|^2f\,dv+\int_{|v-u|\leq R}f\,dv\\
		&\leq {1\over R^2}\int_{{\mathbb{R}^N}}|v-u|^2f\,dv+\int_{|v-u|\leq R}f\,dv\\
		&\leq {1\over R^2}N\rho_{f} T_{f}+C  R^N \|f\|_{L^\infty}
	\end{align*}
	for any $R>0$. 	We take $R^{N+2}={\frac{\rho_{f} T_{f}}{\|f\|_{L^\infty}}}$ to obtain
	\begin{align*}
		{\frac{N}{R^2}}\rho_{f} T_{f}+C  R^N \|f\|_{L^\infty}
		&=(\rho_{f} T_{f})^{{\frac{-2}{N+2}}}\|f\|_{L^\infty}^{{\frac{2}{N+2}}}N\rho_{f} T_{f}
		+C(\rho_{f} T_{f})^{{\frac{N}{N+2}}}\|f\|_{L^\infty}^{{\frac{-N}{N+2}}+1}\cr
		&\leq C(\rho_{f} T_{f})^{{\frac{N}{N+2}}}\|f\|_{L^\infty}^{{\frac{2}{N+2}}}.
	\end{align*}
	Thus, we get
\[
		({\rho_{f}})^{N+2}
		\leq C(\rho_{f} T_{f})^{N}\|f\|_{L^\infty}^{2}.
\]
	Multiplying $(\rho_{f})^{-N}$ and taking square root to both sides yield
\[
		\rho_{f}\leq C\|f\|_{L^{\infty}}({T_{f}})^{N\over2}.
\]
This concludes the desired results.
	\end{proof}

We now provide the proof of Proposition \ref{prop2.2}, using the technique in \cite{CYS19}. We employ Feynman-Kac's formula to show the $L^{\infty}$-estimate. We consider the backward stochastic integral equations and then use Ito's rule, which results in a term with a $v$-derivative. To eliminate this term, we apply expectation and thereby determine the form of $f^{n+1}$.

	\begin{proof}[Proof of Proposition \ref{prop2.2}]
	We first notice that the global-in-time existence and uniqueness of classical solutions to \eqref{eqn2.2} can be obtained by the classical existence theory due to the regularizations. Thus, in the rest of the proof, we only provide the uniform bound estimates. 
		\noindent\newline
		$\bullet$ ($L^\infty$-estimate) Let us consider the backward stochastic integral equations:
		\begin{align*}
			&X^{n+1}(s;x,v,t)
			=x-\int_{s}^{t}V^{n+1}(\tau;x,v,t)\,d\tau,\cr
			&V^{n+1}(s;x,v,t)
			=v-\int_{s}^{t}\Big( u^\e_{f^n}(X^{n+1}(\tau;x,v,t),\tau)-V^{n+1}(\tau;x,v,t) \Big) \,d\tau-\int_{s}^{t}\sqrt{2T_{f^n}^{\epsilon, \delta}(\tau;x,v,t),\tau)}\,dB_\tau.
		\end{align*}
		Due to the regularizations, there exists a unique strong solution to the above equations. For simplicity, we set $Z^{n+1}(s):= (X^{n+1}(s), V^{n+1}(s)):= (X^{n+1}(s;x,v,t), V^{n+1}(s;x,v,t))$. Then, applying Ito's product rule, we obtain
\begin{align*}
\begin{aligned}
&d(e^{-Ns} f^{n+1}(Z^{n+1}(s),s)) \cr
&\quad =  -Ne^{-Ns} f^{n+1}(Z^{n+1}(s),s)\,ds +   e^{-Ns} \pa_s f^{n+1}(Z^{n+1}(s),s)\,ds + e^{-Ns}\left\langle V^{n+1},\nabla_xf^{n+1}(Z^{n+1}(s),s)\right\rangle ds \cr
&\qquad +  e^{-Ns}\left\langle \left(u^\e_{f^n}(X^{n+1}(s),s) - V^{n+1} \right),\nabla_vf^{n+1}(Z^{n+1}(s),s)\right\rangle ds \cr
&\qquad + e^{-Ns} T_{f^n}^{\epsilon, \delta}(X^{n+1}(s),s)\Delta_vf^{n+1}(Z^{n+1}(s),s)\, ds\cr
&\qquad +  e^{-Ns}\sqrt{2T_{f^n}^{\epsilon, \delta}(X^{n+1}(s),s)} \left\langle \nabla_vf^{n+1}(Z^{n+1}(s),s),  dB_s\right\rangle \cr
&\quad =   e^{-Ns}\sqrt{2T_{f^n}^{\epsilon, \delta}(X^{n+1}(s),s)}  \left\langle \nabla_vf^{n+1}(Z^{n+1}(s),s), dB_s\right\rangle 
\end{aligned}
\end{align*}
due to \eqref{eqn2.2}. Here $\left\langle \cdot,\cdot\right\rangle $ denotes the Euclidean inner product. Note that $f^{n+1}(Z^{n+1}(t),t) = f^{n+1}(x,v,t)$, and thus by integrating the above equation over $[0,t]$, we obtain
\[
e^{-Nt} f^{n+1}(x,v,t) = f_{0,\e}(Z^{n+1}(0))   + \int_0^t e^{-Ns}\sqrt{2T_{f^n}^{\epsilon, \delta}(X^{n+1}(s),s)}  \left\langle \nabla_vf^{n+1}(Z^{n+1}(s),s), dB_s\right\rangle.  
\]
Then taking the expectation of the above gives
\bq\label{mild_form}
f^{n+1}(x,v,t) = e^{Nt} \mathbb{E}\left[f_{0,\e}(Z^{n+1}(0))\rt],
\eq
where we used
\begin{align*}
\mathbb{E}\left[\int_0^t e^{-Ns}\sqrt{2T_{f^n}^{\epsilon, \delta}(X^{n+1}(s),s)}  \left\langle \nabla_vf^{n+1}(Z^{n+1}(s),s), dB_s\right\rangle \right]=0
\end{align*}				
		and from which we first easily have
		\[
		\sup_{0 \leq t \leq T}\sup_{n \in \N}\|f^n(t)\|_{L^1\cap L^\infty} \leq \|f_{0,\e}\|_{L^1\cap L^\infty} e^{NT}.
		\]
		\noindent\newline
		$\bullet$ ($L^\infty_q$-estimate)  A direct computation gives
		\begin{align*}
			d\left( e^sV^{n+1}(s) \right) 
			&=e^s V^{n+1}(s)\,ds
			+e^s\left( \big(  u^\e_{f^n}(X^{n+1}(s),s)-V^{n+1}(s)\big) \,ds+\sqrt{2T_{f^n}^{\epsilon, \delta}(X^{n+1}(s),s)}\,dB_s\right)\cr
			&= e^s  u^\e_{f^n}(X^{n+1}(s),s)\,ds
			+e^s\sqrt{2T_{f^n}^{\epsilon, \delta}(X^{n+1}(s),s)}\,dB_s.
		\end{align*}
		Integrating over $[0,t]$, we get
\[
			e^t v
			=V^{n+1}(0)
			+\int_{0}^{t} e^\tau u^\e_{f^n}(X^{n+1}(\tau),\tau)\, d\tau
			+\int_{0}^{t} e^\tau\sqrt{2T_{f^n}^{\epsilon, \delta}(X^{n+1}(\tau),\tau)}\,dB_\tau.
\]
		Then, by Lemma \ref{lemma2.1}, we obtain
		\[
		|v| \leq |V^{n+1}(0)| + C_{\e,\delta} \leq C_{\e,\delta}(1+ |V^{n+1}(0)|).
		\] 
		This, combined with \eqref{mild_form} gives 
		\[
		(1 + |v|^q) f^{n+1}(x,v,t) \leq C_{\e,\delta} \mathbb{E}\left[(1+ |V^{n+1}(0)|^q)f_{0,\e}(Z^{n+1}(0))\right] \leq C_{\e,\delta}\|f_{0,\e}\|_{L^\infty_q}.
		\]
		Hence, we have
		\[
		\sup_{0 \leq t \leq T}\sup_{n \in \N}\|f^n(t)\|_{L^\infty_q} \leq C_{\e,\delta}\|f_{0,\e}\|_{L^\infty_q}.
		\]
		\noindent\newline
		$\bullet$ (Lower bound estimate on $\rho^n$): Using the assumption $f_{0,\e}\geq \epsilon {e^{-|v|^2}}$, we observe from \eqref{mild_form} that 
		\begin{align*}
			f^{n+1}(x,v,t)
			\geq \e e^{Nt}\mathbb{E}\left[   {e^{-|V^{n+1}(0)|^2}}\right].
		\end{align*}
		On the other hand, similarly as before, we easily get 
		\[
		|V^{n+1}(0)| \leq C_{\e,\delta} (1 + |v|),
		\]
		and this gives
\[
			f^{n+1}(x,v,t)
			\geq \epsilon e^{Nt} \mathbb{E}\left[  {e^{-|V^{n+1}(0)|^2}}\right]
			\geq \epsilon e^{Nt} {e^{-C_{\e,\delta}(1+|v|)^2}}.
\]
This shows the strict positivity of $f^n$ for all $n \in \N$. Moreover, we have
		\begin{align*}
			\rho_{f^n}
			\geq\int_{{\mathbb{R}^N}}f^n(x,v,t)\,dv
			\geq \int_{{\mathbb{R}^N}}e^{Nt}\epsilon {e^{-C_{\e,\delta}(1+|v|)^2}}dv
			\geq c_{\e,\delta},
		\end{align*}
which completes the proof.
 	\end{proof}
 	
	Note that there is a relationship between $\rho_f$ and $T_f$ stated in Lemma \ref{lemma2.1}. Using that, we next show the lower bound estimate on $T_{f^n}^{\epsilon, \delta}$ uniformly in $n$.

	\begin{lemma}\label{lem_low}
	Let $q>N+2$. Then there exists $c_{\e,\delta} > 0$ such that 
		\begin{align*}
			\inf_{(x,t) \in \T^N \times (0,T)} \inf_{n \in \N} T_{f^n}^{\epsilon, \delta} \geq c_{\e,\delta}.
		\end{align*}
	\end{lemma}
	\begin{proof}
		By using Lemma \ref{lemma2.1} and Proposition \ref{prop2.2}, we find
		\begin{align*}
			c_{\e,\delta} \leq \rho_{f^n}
			\leq C\|f^n\|_{L^{\infty}}\left( T_{f^n}\right)^{\frac{N}{2}}\leq C\left( T_{f^n}\right)^{\frac{N}{2}},
		\end{align*}
		which implies
		\begin{align*}
			((\rho_{f^n} T_{f^n})*\theta_\epsilon)(x)
			\geq c_{\e,\delta}\int_{{\mathbb{T}^N}}\theta_\epsilon(y)\,dy
			= c_{\e,\delta}.
		\end{align*}
	Thus we have from Remark \ref{rem_poT} that 
			\[
	\Phi_{f_\epsilon}^{\e,\delta} \geq N \rho_{f_\epsilon} T_{f_\epsilon} * \theta_\epsilon \geq c_{\e,\delta}.
	\]	
Combining all of the above estimates together with Lemma \ref{lemma2.1} and Proposition \ref{prop2.2},  we conclude
		\begin{align*}
			T_{f^n}^{\epsilon, \delta}(x,t)=
			\frac{\Phi_{f^n}^{\e,\delta} + \delta^2 }{N\rho_{f^n}*\theta_\epsilon+\delta\left( 1+\Phi_{f^n}^{\e,\delta} \right) } \geq \frac{\Phi_{f^n}^{\e,\delta}  }{ C_{\e,\delta, N}\|f_{0,\e}\|_{L^\infty_q}+\delta\left( 1+\Phi_{f^n}^{\e,\delta} \right) }
			\geq c_{\e,\delta},
		\end{align*}
		where we used the fact that the function $f: \R_+ \to \R$ given by $f(x) = \frac{x}{a+bx}$ with $a,b > 0$ is increasing. This completes the proof.
	\end{proof}
	
By using the lower bound estimate on $T_{f^n}^{\e, \delta}$ obtained in the above lemma, we further obtain the uniform-in-$n$ bound estimate on $\|\nabla_vf^n\|_{L^2(0,T; L_q^2)}$ in the following lemma. 

		\begin{lemma}		\label{lq2}
		Let $q>N+2$. Then, there exists a positive constant $C_{\epsilon,\delta}$ independent of $n$ such that
\[
\sup_{n \in \N} \sup_{0 \leq t \leq T}\left(\|f^n(t)\|_{L_q^2}^2
+ \|\nabla_{x}f^n(t)\|_{L_q^2}^2 + \|\nabla_{v}f^n(t)\|_{L_q^2}^2\right) 
\leq 
C_{\e,\delta}\left(\|f_{0,\e}\|_{L_q^2}^2+\|\nabla_{x}f_{0,\e}\|_{L_q^2}^2+\|\nabla_{v}f_{0,\e}\|_{L_q^2}^2\right).
\]
	\end{lemma}
	\begin{proof}
 Since the proof involves lengthy computations, only the estimate for $\|f^n(t)\|_{L_q^2}^2$ is covered here for brevity. The rest of the proof is provided in Appendix \ref{app_a}.
		
	Multiplying \eqref{eqn2.2} by $(1+|v|^q)f$ and integrating  over $(x,v)$, one finds
		\begin{align*}
			&{d\over dt}\iint_{\mathbb{T}^N\times\mathbb{R}^N}(1+|v|^q)|f^{n+1}|^2\,dvdx\cr
			&\quad =2\iint_{\mathbb{T}^N\times\mathbb{R}^N}(1+|v|^q)f^{n+1}T_{f^n}^{\e,\delta}\Delta_vf^{n+1}\,dvdx
			-2\iint_{\mathbb{T}^N\times\mathbb{R}^N}(1+|v|^q)f^{n+1}\nabla_v\cdot\lt((u_{f^n}^\e-v)f^{n+1} \rt)dvdx\cr
			&\quad =:\rom{1}+\rom{2}.
		\end{align*}
		$\bullet$ Estimate of $\rom{1}$:
		Using integration by parts, we get
		\begin{align*}
			\rom{1}
			&=-2\iint_{\mathbb{T}^N\times\mathbb{R}^N}\nabla_v \Big((1+|v|^q)f^{n+1}\Big)T_{f^n}^{\e,\delta}\cdot\nabla_vf^{n+1}\,dvdx\cr
			&=-2\iint_{\mathbb{T}^N\times\mathbb{R}^N}q|v|^{q-2}f^{n+1}T_{f^n}^{\e,\delta}v\cdot\nabla_vf^{n+1}\,dvdx
			-2\iint_{\mathbb{T}^N\times\mathbb{R}^N}(1+|v|^q)T_{f^n}^{\e,\delta}|\nabla_vf^{n+1}|^2\,dvdx\cr
			&=:\rom{1}_1+\rom{1}_2,
		\end{align*}
		where  
			\begin{align*}
			\rom{1}_1
			&=-\iint_{\mathbb{T}^N\times\mathbb{R}^N}qv|v|^{q-2}T_{f^n}^{\e,\delta}\cdot\nabla_v|f^{n+1}|^2\,dvdx\cr
			&=\iint_{\mathbb{T}^N\times\mathbb{R}^N}qN|v|^{q-2}T_{f^n}^{\e,\delta}|f^{n+1}|^2\,dvdx
			+\iint_{\mathbb{T}^N\times\mathbb{R}^N}q(q-2)|v|^{q-2}T_{f^n}^{\e,\delta}|f^{n+1}|^2\,dvdx\cr
			&\leq q(N+q-2)\|T_{f^n}^{\e,\delta}\|_{L^\infty}\iint_{\mathbb{T}^N\times\mathbb{R}^N}(1+|v|^q)|f^{n+1}|^2\,dvdx\cr
			&\leq q(N+q-2)C_\delta\left\|f^{n+1}\right\|_{L_q^2}^2.
		\end{align*}
		For $I_2$, we use Lemma \ref{lem_low} to deduce 
		\[
		I_2 \leq -2 c_{\e,\delta}\left\|\nabla_vf^{n+1}\right\|_{L_q^2}^2.
		\]

\noindent		$\bullet$ Estimate of $\rom{2}$:
		A direct computation yields
		\begin{align*}
			\rom{2}
			&=2\iint_{\mathbb{T}^N\times\mathbb{R}^N}(1+|v|^q)f^{n+1}\Big( Nf^{n+1}-(u_{f^n}^\e-v)\cdot\nabla_vf^{n+1}\Big)\,dvdx =:\rom{2}_1+\rom{2}_2,
		\end{align*}
		where we readily obtain
		\[
		II_1 = 2N\left\|f^{n+1}\right\|_{L_q^2}^2.
		\]
		In the same manner as in $\rom{1}_1$, we get
		\begin{align*}
			\rom{2}_2&=\iint_{\mathbb{T}^N\times\mathbb{R}^N}\Big( qv|v|^{q-2}\cdot(u_{f^n}^\e-v)+(1+|v|^q)(-N)\Big)|f^{n+1}|^2\,dvdx\cr
			&\leq q\iint_{\mathbb{T}^N\times\mathbb{R}^N}|u_{f^n}^\e||v|^{q-1}|f^{n+1}|^2\,dvdx
			+(q-N)\iint_{\mathbb{T}^N\times\mathbb{R}^N}(1+|v|^q)|f^{n+1}|^2\,dvdx\cr
			&\leq q\|u_{f^n}^\e\|_{L^\infty}\iint_{\mathbb{T}^N\times\mathbb{R}^N}(1+|v|^q)|f^{n+1}|^2\,dvdx
			+(q-N)\iint_{\mathbb{T}^N\times\mathbb{R}^N}(1+|v|^q)|f^{n+1}|^2\,dvdx\cr
			&\leq (qC_\epsilon+q-N)\| f^{n+1}\|^2_{L^2_q}.
		\end{align*}
In the last line, we used Lemma \ref{lemma2.1}. Combining the above estimates gives
\[
\frac{d}{dt}\|f^n(t)\|_{L_q^2}^2 +  c_{\e,\delta}\left\|\nabla_vf^{n+1}\right\|_{L_q^2}^2 \leq C_{\e,\delta,q,N} \|f^n(t)\|_{L_q^2}^2
\]
and applying the Gr\"onwall's lemma deduces the zeroth-order estimate. For the first-order derivative estimates, we obtain from Appendix \ref{app_a} that 
\begin{align*}
\begin{split}
&\frac{d}{dt}\left(\|\nabla_{x}f^{n+1}(t)\|_{L_q^2}^2
+\|\nabla_{v}f^{n+1}(t)\|_{L_q^2}^2\right)
+c_{\epsilon,\delta}
\|\nabla_{x}\nabla_{v}f^{n+1}\|_{L_q^2}^2 \cr
&\quad \leq C_{\e,\delta,q,N}
\left(\|\nabla_{x}f^{n+1}(t)\|_{L_q^2}^2
+\|\nabla_{v}f^{n+1}(t)\|_{L_q^2}^2\right).
\end{split}
\end{align*}
Hence, by applying Gr\"onwall's lemma to the above, we conclude the desired result. 
	\end{proof}
	\begin{remark}Note that for $q>N$, $f \in L^\infty_q (\T^N \times \R^N)$ implies $f \in L^2_q(\T^N \times \R^N)$. Indeed, we observe
	\[
	\intrr (1+|v|^q)f^2\,dvdx \leq \|f\|_{L^\infty_q}^2 \intrr (1+|v|^q)^{-1}\,dvdx.
	\]
	However, Lemma \ref{lq2} provides the uniform-in-$n$ bound estimate on $\|\nabla_vf^n\|_{L^\infty(0,T; L_q^2)}$ which will be crucially used to show that $\{f^n\}$ is a Cauchy sequence in $L^\infty(0,T;L^2_q(\T^N\times\R^N))$.
	\end{remark}

	%
	%
	%
	%
	%
	%
	%
	%
	%
	\section{Weak solutions to the regularized equation}\label{sec:reg_weak}
	In this section, we establish the existence of weak solutions to the regularized equation \eqref{eqn2.2} in the sense of Definition \ref{def1.1} and provide the bound estimate of the third velocity-moment on $f_{\epsilon,\delta}$ uniformly in $\epsilon$ and $\delta$.
	%
	%
	%
	%
	%
	%
	%
	%
	\subsection{Existence of weak solutions}

	%
	%
	%
	%
	%
	%
	%
	%
\subsubsection{Cauchy estimates}\label{ssec_cauchy} In this part, we show that $\{f^n\}$ is Cauchy  in $L^{\infty}(0,T;L_q^2(\mathbb{T}^N\times\mathbb{R}^N))$. For this, we first present the following auxiliary lemmas. 
	\begin{lemma}\emph{\cite{BDM09}}
		\label{cau1}
	Let $\{a_n\}_{n\in \mathbb{N}}$ be a sequence of nonnegative continuous functions defined on $[0,T]$ satisfying
		\begin{equation*}
			a_{n+1}(t)\leq C_1+C_2\int_{0}^{t}a_n(\tau)\,d\tau+C_3\int_{0}^{t}a_{n+1}(\tau)\,d\tau,\;\; 0\leq t\leq T,
		\end{equation*}
		where $C_i (i=1,2,3)$ are nonnegative constants. Then there exists a positive constant $K$ such that for all $n\in \mathbb{N}$
		\begin{equation*}
			a_n(t) \leq \begin{cases}
				{K^nt^n\over n !} &  \mbox{if} \;\;C_1=0,\\
				Ke^{Kt} &  \mbox{if} \;\;C_1>0.\\
			\end{cases}
		\end{equation*}
	\end{lemma}
	\begin{lemma}
		\label{lem4.1}
		Let $q>N+4$. Then we have, 
		\begin{align*}
			\lt\|u_{f^{n+1}}^{\epsilon}-u_{f^n}^{\epsilon}\rt\|_{L^\infty} + \lt\|T_{f^{n+1}}^{\epsilon, \delta}-T_{f^n}^{\epsilon, \delta}\rt\|_{L^\infty}
			\leq C_{\epsilon, \delta,T}\left\|f^{n+1}-f^n\right\|_{L^{\infty}(0,T;L_q^2(\mathbb{T}^N\times\mathbb{R}^N))}^2.
		\end{align*}
		
	\end{lemma}
	\begin{proof} 	For $q > N+4$, we first observe 
		\begin{align*}
		\begin{aligned}
			\lt|\rho_{f^{n+1}}-\rho_{f^n}\rt|
			&=\left|\int_{\mathbb{R}^N}\big(1+|v|^q\big)^{-\frac 12}\big(1+|v|^q\big)^{\frac 12}(f^{n+1}-f^n)\, dv\right| \leq C\left\|\big(1+|v|^q\big)^{\frac 12}(f^{n+1}-f^n)\right\|_{L^2(\mathbb{R}_v^N)}
		\end{aligned}
		\end{align*}
		and
		\begin{align*}
		\begin{aligned}
			\lt|\rho_{f^{n+1}}u_{f^{n+1}}-\rho_{f^n}u_{f^n}\rt|
			&=\left|\int_{\mathbb{R}^N}v\big(1+|v|^q\big)^{-\frac 12}\big(1+|v|^q\big)^{\frac12}(f^{n+1}-f^n)\, dv\right|\cr
			&\leq C\left\|\big(1+|v|^q\big)^{\frac 12}(f^{n+1}-f^n)\right\|_{L^2(\mathbb{R}_v^N)}.
		\end{aligned}
		\end{align*}
Then, we obtain
		\begin{align*}
			\lt|(\rho_{f^{n+1}}-\rho_{f^n})*\theta_\epsilon\rt| \leq \|\rho_{f^{n+1}}-\rho_{f^n}\|_{L^2(\mathbb{T}^N)}\|\theta_\e\|_{L^2(\mathbb{T}^N)} 
			&\leq C_\epsilon\left\|f^{n+1}-f^n\right\|_{L^{\infty}(0,T;L_q^2(\mathbb{T}^N\times\mathbb{R}^N))}
		\end{align*}
		and 
		\begin{align*}
			\lt|(\rho_{f^{n+1}} u_{f^{n+1}}-\rho_{f^n} u_{f^n})*\theta_\epsilon\rt| &\leq \|\rho_{f^{n+1}} u_{f^{n+1}}-\rho_{f^n} u_{f^n}\|_{L^2(\mathbb{T}^N)}\|\theta_\e\|_{L^2(\mathbb{T}^N)} \\
			&\leq C_\epsilon\left\|f^{n+1}-f^n\right\|_{L^{\infty}(0,T;L_q^2(\mathbb{T}^N\times\mathbb{R}^N))}.
		\end{align*}
		This together with Lemma \ref{lemma2.1} and Proposition \ref{prop2.2} gives
		\begin{align}\label{u ee}\begin{split}
				\Big|u_{f^{n+1}}^{\epsilon}-u_{f^n}^{\epsilon}\Big|
				&\leq\frac{\left|\big(\rho_{f^{n+1}} u_{f^{n+1}}-\rho_{f^n} u_{f^n}\big)*\theta_\epsilon\right|}
				{\rho_{f^{n+1}}*\theta_\epsilon+\epsilon\left(1+\big|\big(\rho_{f^{n+1}} u_{f^{n+1}}\big)*\theta_\epsilon\big|^2\right)}\cr
				&\quad +\frac{\left|\big(\rho_{f^n} u_{f^n}\big)*\theta_\epsilon\right|\Big|\big(\rho_{f^n}-\rho_{f^{n+1}}\big)*\theta_\epsilon+\epsilon\Big(\big|\rho_{f^n}u_{f^n}*\theta_\epsilon\big|^2-\big|\rho_{f^{n+1}}u_{f^{n+1}}*\theta_\epsilon\big|^2\Big)\Big|}
				{\left\lbrace \rho_{f^{n+1}}*\theta_\epsilon+\epsilon\Big(1+\big|\big(\rho_{f^{n+1}} u_{f^{n+1}}\big)*\theta_\epsilon\big|^2\Big)\right\rbrace \left\lbrace \rho_{f^n}*\theta_\epsilon+\epsilon\left(1+\left|\left(\rho_{f^n} u_{f^n}\right)*\theta_\epsilon\right|^2\right)\right\rbrace }\cr
				&\leq C_\e \left|\big(\rho_{f^{n+1}} u_{f^{n+1}}-\rho_{f^n} u_{f^n}\big)*\theta_\epsilon\right|
				+C_\e \left|\big(\rho_{f^n}-\rho_{f^{n+1}}\big)*\theta_\epsilon\right|\cr
				&\quad + C_\e \Big(\big|\rho_{f^n} u_{f^n}* \theta_\e\big|+\big|\rho_{f^{n+1}}u_{f^{n+1}}*\theta_\e\big|\Big)\left|\big(\rho_{f^n} u_{f^n}-\rho_{f^{n+1}} u_{f^{n+1}}\big)*\theta_\epsilon\right|\cr
				&\leq C_\epsilon\left\|f^{n+1}-f^n\right\|_{L^{\infty}(0,T;L_q^2(\mathbb{T}^N\times\mathbb{R}^N))}.
		\end{split}	\end{align}
		Thus, we deduce
		\begin{align*}
			\|u_{f^{n+1}}^{\epsilon}-u_{f^n}^{\epsilon}\|_{L^\infty}
			\le C_\epsilon\left\|f^{n+1}-f^n\right\|_{L^{\infty}(0,T;L_q^2(\mathbb{T}^N\times\mathbb{R}^N))}.
		\end{align*}		
For the estimate of $T_{f^{n+1}}^{\epsilon, \delta}-T_{f^n}^{\epsilon, \delta}$, we use Proposition \ref{prop2.2} to get		
		\begin{align*}
			\left|u_{f^n}\right| = \frac{\lt|\rho_{f^n} u_{f^n}\rt|}{\rho_{f^n}}
			\leq C_{\epsilon, \delta}\int_{{\mathbb{R}^N}}|v|\big(1+|v|^q\big)^{-1}\big(1+|v|^q\big)|f^n|\, dv
			\leq C_{\epsilon, \delta}\left\|f^n\right\|_{L_q^{\infty}}\int_{{\mathbb{R}^N}}|v|\big(1+|v|^q\big)^{-1}\, dv
			\leq C_{\epsilon, \delta}.
		\end{align*}
We also find 
		\begin{align*}
			\Phi_{f^{n+1}}^{\epsilon, \delta}-\Phi_{f^n}^{\epsilon, \delta}
			&=\left( \int_{{\mathbb{R}^N}}|v|^2\left( f^{n+1}-f^n\right) \, dv\right) * \theta_\epsilon
			- \frac{\big| \rho_{f^{n+1}} u_{f^{n+1}}*\theta_\epsilon\big|^2-\left| \rho_{f^n} u_{f^n}*\theta_\epsilon\right|^2}{\rho_{f^{n+1}}*\theta_\epsilon+\delta\left( 1+\big| \rho_{f^{n+1}} u_{f^{n+1}}*\theta_\epsilon\big|^2\right)  }\cr
			&\quad +\frac{\left| \rho_{f^n} u_{f^n}*\theta_\epsilon\right|^2\left\lbrace \big( \rho_{f^{n+1}}-\rho_{f^n}\big)*\theta_\epsilon +\delta\left( \big| \rho_{f^{n+1}} u_{f^{n+1}}*\theta_\epsilon\big|^2-\left| \rho_{f^n} u_{f^n}*\theta_\epsilon\right|^2\right)   \right\rbrace}{\left\lbrace  \rho_{f^n}*\theta_\epsilon+\delta\left( 1+\left| \rho_{f^n} u_{f^n}*\theta_\epsilon\right|^2\right) \right\rbrace \left\lbrace  \rho_{f^{n+1}}*\theta_\epsilon+\delta\left( 1+\big| \rho_{f^{n+1}} u_{f^{n+1}}*\theta_\epsilon\big|^2\right) \right\rbrace  }.
		\end{align*}
		Hence, similarly to \eqref{u ee}, we have
		\begin{align*}
			\big|\Phi_{f^{n+1}}^{\epsilon, \delta}-\Phi_{f^n}^{\epsilon, \delta}\big|
			\leq C_{\epsilon, \delta}\left\|f^{n+1}-f^n\right\|_{L^{\infty}(0,T;L_q^2(\mathbb{T}^N\times\mathbb{R}^N))}.
		\end{align*}
We further estimate
		\begin{align}\label{diff_tn}
		\begin{aligned}
			\big|T_{f^{n+1}}^{\epsilon, \delta}-T_{f^n}^{\epsilon, \delta}\big
			|
			&=\left|\frac{\Phi_{f^{n+1}}^{\epsilon, \delta} + \delta^2}{N\rho_{f^{n+1}}*\theta_\epsilon+\delta\big( 1+\Phi_{f^{n+1}}^{\epsilon, \delta} \big)}-\frac{\Phi_{f^n}^{\epsilon, \delta} + \delta^2}{N\rho_{f^n}*\theta_\epsilon+\delta\big( 1+\Phi_{f^n}^{\epsilon, \delta} \big)}\right|\cr
			&\leq{\left|\frac{\Phi_{f^{n+1}}^{\epsilon, \delta}-\Phi_{f^n}^{\epsilon, \delta}}{N\rho_{f^{n+1}}*\theta_\epsilon+\delta\big(1+\Phi_{f^{n+1}}^{\epsilon, \delta}\big)}\right| }\cr
			&\quad + (\Phi_{f^n}^{\epsilon, \delta} + \delta^2)\frac{N \left|\big(\rho_{f^n}-\rho_{f^{n+1}}\big)*\theta_\epsilon\right|+\delta \left|\Phi_{f^n}^{\epsilon, \delta} - \Phi_{f^{n+1}}^{\epsilon, \delta}\right|}
			{\left\lbrace \rho_{f^{n+1}}*\theta_\epsilon+\delta\big(1+\Phi_{f^{n+1}}^{\epsilon}\big)\right\rbrace \left\lbrace \rho_{f^n}*\theta_\epsilon+\delta\big(1+\Phi_{f^n}^{\epsilon}\big)\right\rbrace }\cr
			&\leq C_{\epsilon, \delta}\|f^{n+1}-f^n\|_{L^{\infty}(0,T;L_q^2(\mathbb{T}^N\times\mathbb{R}^N))},
		\end{aligned}
		\end{align}
		which completes the proof.
	\end{proof}

We then present the Cauchy estimate of our approximate solution sequence $\{f^n\}$ in the lemma below.	
\begin{lemma}\label{lem_cau} Let $q>N+4$ and $f^n \in L^{\infty}(0,T;L_q^{\infty}(\mathbb{T}^N\times\mathbb{R}^N))$ be the solution to the equation \eqref{eqn2.2} constructed in Proposition \ref{prop2.2}. Then $\{f^n\}$ is a Cauchy sequence in $L^{\infty}(0,T;L_q^2(\mathbb{T}^N\times\mathbb{R}^N))$, and thus there exists $f \in L^{\infty}(0,T;L_q^2(\mathbb{T}^N\times\mathbb{R}^N))$ such that
	\begin{align*}
		\sup_{0\leq t\leq T}\|(f^n-f)(t)\|_{L_q^2}\rightarrow 0
		\;\;\;
		\mbox{as $n\rightarrow \infty$.}
	\end{align*}
\end{lemma}	
 \begin{proof}
 	For this, we observe from  \eqref{eqn2.2} that
	\begin{align*}\begin{split}
	 \partial_t\left(f^{n+1}-f^n\right)+v\cdot\nabla_x\left(f^{n+1}-f^n\right) 
			&  =T_{f^n}^{\e,\delta}\Delta_v\left(f^{n+1}-f^n\right)+\big(T_{f^n}^{\e,\delta}-T_{f^{n-1}}^{\e,\delta}\big)\Delta_vf^n +N\left(f^{n+1}-f^n\right)\cr
			&\quad +\big(v-u_{f^n}^{\epsilon}\big)\cdot\nabla_v\left(f^{n+1}-f^n\right)
			+ \big(u_{f^n}^{\epsilon}-u_{f^{n-1}}^{\epsilon}\big)\cdot\nabla_vf^n. 
	\end{split}	\end{align*}
Multiplying it by $(1+|v|^q)(f^{n+1}-f^n)$ and integrating over $(x,v)$, we get
	\begin{align*}
		&\frac12{d\over dt}\iint_{\mathbb{T}^N\times\mathbb{R}^N}\big(1+|v|^q\big)\left|f^{n+1}-f^n\right|^2\, dvdx\cr
		&\quad = \iint_{{\mathbb{T}^N}\times\R^N}\big(1+|v|^q\big)\left(f^{n+1}-f^n\right) \big(u_{f^n}^{\epsilon}-u_{f^{n-1}}^{\epsilon}\big) \cdot\nabla_vf^n\,dvdx\cr
		&\qquad +\iint_{{\mathbb{T}^N}\times\R^N}\big(1+|v|^q\big)\left(f^{n+1}-f^n\right)\left[ T_{f^n}^{\e,\delta}\Delta_v\left(f^{n+1}-f^n\right)+\big(T_{f^n}^{\e,\delta}-T_{f^{n-1}}^{\e,\delta}\big)\Delta_vf^n\right]dvdx\cr
		&\qquad +\iint_{{\mathbb{T}^N}\times\R^N}\big(1+|v|^q\big)\left(f^{n+1}-f^n\right)\left[ \big(v-u_{f^n}^{\epsilon}\big)\cdot\nabla_v\left(f^{n+1}-f^n\right)+N\left(f^{n+1}-f^n\right)\right] dvdx\cr
		&\quad =: \rom{1}+\rom{2}+\rom{3}+\rom{4}+\rom{5}.
	\end{align*}
	$\bullet$ Estimate of $\rom{1}$: We use Lemma \ref{lem4.1} to have
	\begin{align*}
		\rom{1}
		&\leq C\|u_{f^n}^{\epsilon}-u_{f^{n-1}}^{\epsilon}\|_{L^\infty}\iint_{\mathbb{T}^N\times\mathbb{R}^N}\big(1+|v|^q\big)\left|f^{n+1}-f^n\right|\big|\nabla_vf^n\big|\,dvdx\cr
		&\leq C_{\e,\delta}\left\|f^n-f^{n-1}\right\|_{L_q^2}\left\|f^{n+1}-f^n\right\|_{L_q^2} \|\nabla_vf^n\|_{L_q^2}\cr
		&\leq C_{\e,\delta}\left( \left\|f^n-f^{n-1}\right\|_{L_q^2}^2+\|\nabla_vf^n\|_{L_q^2}^2\|f^{n+1}-f^n\|_{L_q^2}^2\right).
	\end{align*}
	$\bullet$ Estimate of $\rom{2}$: It follows from Lemma \ref{lemma2.1} that
	\begin{align*}
		\rom{2}
		&=-\frac 12\iint_{\mathbb{T}^N\times\mathbb{R}^N}qv|v|^{q-2}T_{f^n}^{\epsilon, \delta}\cdot\nabla_v\left|f^{n+1}-f^n\right|^2  dvdx
		-\iint_{\mathbb{T}^N\times\mathbb{R}^N}\big(1+|v|^q\big)T_{f^n}^{\epsilon, \delta}\left|\nabla_v\left(f^{n+1}-f^n\right)\right|^2  dvdx\cr
		&\leq C\iint_{\mathbb{T}^N\times\mathbb{R}^N}|v|^{q-2}T_{f^n}^{\epsilon, \delta}\left|f^{n+1}-f^n\right|^2 dvdx
		- c_{\e,\delta}\iint_{\mathbb{T}^N\times\mathbb{R}^N}\big(1+|v|^q\big)\left|\nabla_v\left(f^{n+1}-f^n\right)\right|^2  dvdx\\
		&\leq C_\delta\left\|f^{n+1}-f^n\right\|_{L_q^2}^2	 -c_{\e,\delta}\left\| \nabla_v(f^{n+1}-f^n)\right\|_{L_q^2}^2.
	\end{align*}
	$\bullet$ Estimate of $\rom{3}$: Similarly to the estimate of $\rom{1}$, we find
	\begin{align*}
		\rom{3}
		&= -\iint_{\mathbb{T}^N\times\mathbb{R}^N}qv|v|^{q-2}\left(f^{n+1}-f^n\right)\big(T_{f^n}^{\epsilon, \delta}-T_{f^{n-1}}^{\e,\delta}\big)\cdot\nabla_vf^n\, dvdx\cr
		&\quad -\iint_{\mathbb{T}^N\times\mathbb{R}^N}\big(1+|v|^q\big)\nabla_v\left(f^{n+1}-f^n\right)\big(T_{f^n}^{\epsilon, \delta}-T_{f^{n-1}}^{\e,\delta}\big)\cdot\nabla_vf^n\,dvdx\cr
		&\leq C\lt\|T_{f^n}^{\epsilon, \delta}-T_{f^{n-1}}^{\e,\delta}\rt\|_{L^\infty}\iint_{\mathbb{T}^N\times\mathbb{R}^N}\big(1+|v|^q\big)\left|f^{n+1}-f^n\right|\big|\nabla_vf^n\big|\,dvdx\cr
		&\quad +C\lt\|T_{f^n}^{\epsilon, \delta}-T_{f^{n-1}}^{\e,\delta}\rt\|_{L^\infty}\iint_{\mathbb{T}^N\times\mathbb{R}^N}\big(1+|v|^q\big)\left|\nabla_v\left(f^{n+1}-f^n\right)\right|\big|\nabla_vf^n\big|\,dvdx\cr
		&\leq C_{\e,\delta}\left\|f^n-f^{n-1}\right\|_{L_q^2}\left\|f^{n+1}-f^n\right\|_{L_q^2} \|\nabla_vf^n\|_{L_q^2}\cr
		&\quad +C_{\e,\delta}\left\|f^n-f^{n-1}\right\|_{L_q^2}\left\| \nabla_v(f^{n+1}-f^n)\right\|_{L_q^2} \|\nabla_vf^n\|_{L_q^2}\cr
		&\leq C_{\e,\delta}\left( \|\nabla_vf^n\|_{L_q^2}^2+1\right) \left( \left\|f^n-f^{n-1}\right\|_{L_q^2}^2+\left\|f^{n+1}-f^n\right\|_{L_q^2}^2\right)
		+ \frac {c_{\e,\delta}}{4} \left\| \nabla_v(f^{n+1}-f^n)\right\|_{L_q^2}^2.
	\end{align*}
	$\bullet$ Estimate of $\rom{4}$: By using the integration by parts and Lemma \ref{lemma2.1}, we get
	\begin{align*}
		\rom{4}
		&=\frac12\iint_{\mathbb{T}^N\times\mathbb{R}^N}\big(1+|v|^q\big)\big(v-u_{f^n}^{\epsilon}\big)\cdot\nabla_v\left|f^{n+1}-f^n\right|^2dvdx
		\\&=- \frac N2 \iint_{\mathbb{T}^N\times\mathbb{R}^N}\big(1+|v|^q\big)\left|f^{n+1}-f^n\right|^2dvdx
		- \frac q2\iint_{\mathbb{T}^N\times\mathbb{R}^N} v|v|^{q-2}\cdot\big(v-u_{f^n}^{\epsilon}\big)\left|f^{n+1}-f^n\right|^2dvdx\cr
		&\leq C_{\e,\delta}\left\|f^{n+1}-f^n\right\|_{L_q^2}^2.
	\end{align*}
	$\bullet$ Estimate of $\rom{5}$: It is clear that
	\begin{align*}
		\rom{5}
			=N\left\|f^{n+1}-f^n\right\|_{L_q^2}^2.
	\end{align*}
	Combining the above estimates leads to
	\begin{align*}
		{d\over dt}\left\|(f^{n+1}-f^n)(t)\right\|_{L_q^2}^2
		&\leq \left( C_{\e,\delta}\|\nabla_vf^n\|_{L_q^2}^2+C_{\e,\delta}\right) \left\| (f^{n+1}-f^n)(t)\right\|_{L_q^2}^2\cr
		& \quad +\left( C_{\e,\delta}+C_{\e,\delta}\|\nabla_vf^n\|_{L_q^2}^2 \right) \left\|(f^n-f^{n-1})(t)\right\|_{L_q^2}^2\cr
		&\leq C_{\e,\delta}\left(\left\| (f^{n+1}-f^n)(t)\right\|_{L_q^2}^2+\left\|(f^n-f^{n-1})(t)\right\|_{L_q^2}^2 \right).
	\end{align*}
In the last line, we used Lemma \ref{lq2}. Integrating the above over $[0,t]$ implies
	\begin{align*}
		\left\|(f^{n+1}-f^n)(t)\right\|_{L_q^2}^2
		\leq
		C_{\e,\delta}\int_{0}^{t}\left\|(f^{n+1}-f^n)(t)\right\|_{L_q^2}^2 ds
		+C_{\e,\delta}\int_{0}^{t}\left\|(f^{n}-f^{n-1})(t)\right\|_{L_q^2}^2 ds
	\end{align*}
which, together with Lemma \ref{cau1}, gives
	\begin{align*}
		\left\|(f^{n+1}-f^n)(t)\right\|_{L_q^2}^2
		\leq 
		\frac{K^{n+1}T^{n+1}}{(n+1)!}
		,
		\quad
		t\in \left[0, T \right].
	\end{align*}
This concludes that  $\{f^n\}$ is Cauchy  in $L^{\infty}(0,T;L_q^2(\mathbb{T}^N\times\mathbb{R}^N))$.  
\end{proof}	
	%
	%
	%
	%
	%
	%
	%
	%
	\subsubsection{Passing to the limit $n \to \infty$} 
	Now we show that the limiting function $f$ is, in fact, a weak solution of the regularized equation \eqref{eqn2.1} in the sense of Definition \ref{def1.1}:
	\begin{align}\label{weak formulation}\begin{split}
			&-\iint_{\mathbb{T}^N\times\mathbb{R}^N}f_{0,\e}\psi_0\, dvdx
			-\int_{0}^{T}\iint_{\mathbb{T}^N\times\mathbb{R}^N}f\left( \partial_t\psi+v\cdot\nabla_x\psi+(u^\e_f-v)\cdot\nabla_v\psi \right) dvdxdt\\
			&\quad =\int_{0}^{T}\iint_{\mathbb{T}^N\times\mathbb{R}^N} T_{f}^{\e,\delta}f \Delta_v\psi\, dvdxdt,
	\end{split}	\end{align}
	where $\psi(x,v,t)\in \mc_c^2(\mathbb{T}^N\times\mathbb{R}^N\times[0,T])$ \mbox{with} $\psi(x,v,T)=0$.
	Due to the linearity, it suffices to deal with the terms with $u^\epsilon_{f}$ and $T^{\e,\delta}_{f}$ in the above.
	
	By using almost the same argument used in \eqref{u ee}, we first observe 
\[
\|(u^\e_{f^n} - u^\e_f)(t)\|_{L^\infty} \leq C_\e \|(f^n - f)(t)\|_{L^2_q},
\]
and thus by Lemma \ref{lem_cau}, we obtain 
	$$
	u_{f^n}^{\e}\rightarrow u_{f}^\e \quad \mbox{in } L^\infty(\T^N \times (0,T))
	$$
	as $n \to \infty$. On the other hand, we get
\begin{align*}
&\lt|\int_0^T\iint_{{\mathbb{T}^N}\times\mathbb{R}^N} (f^n u^\e_{f^n} - f u^\e_f) \cdot\nabla_v \psi\,dvdxdt \rt| \cr
&\quad \leq \|\nabla_v \psi\|_{L^\infty}\|u^\e_{f^n}\|_{L^\infty} \int_0^T\iint_{{\mathbb{T}^N}\times\mathbb{R}^N} |f^n -f|\,dvdxdt + \|\nabla_v \psi\|_{L^\infty}\int_0^T\iint_{{\mathbb{T}^N}\times\mathbb{R}^N} f|u^\e_{f^n} - u^\e_f| \,dvdxdt\cr
&\quad \leq C\|f^n-f\|_{L_q^2}^2 + C\|u^\e_{f^n} - u^\e_f\|_{L^\infty}
\end{align*}
for some $C>0$ independent of $n$, where we used $\|f\|_{L^1}\le C\|f\|_{L^2_q}$. This shows that 
\[
		\int_0^T\iint_{{\mathbb{T}^N}\times\mathbb{R}^N} f^n u^{\e}_{f^n} \cdot\nabla_v \psi\,dvdxdt \rightarrow \int_0^T\iint_{{\mathbb{T}^N}\times\mathbb{R}^N} f u^\e_{f} \cdot\nabla_v \psi\,dvdxdt 
		\]
as $n\rightarrow \infty.$ Similarly, it follows from \eqref{diff_tn} that 
\[
\|T_{f^n}^{\epsilon, \delta}-T_f^{\epsilon, \delta}\|_{L^\infty} \leq C_{\epsilon, \delta}\left\|f^n-f\right\|_{L_q^2},
\]
and thus we conclude
	$$
	\int_{0}^{T}\iint_{\mathbb{T}^N\times\mathbb{R}^N} T_{f^n}^{\e, \delta}f^n \Delta_v\psi\, dvdxdt\rightarrow \int_{0}^{T}\iint_{\mathbb{T}^N\times\mathbb{R}^N} T^{\e, \delta}_{f}f \Delta_v\psi\, dvdxdt 
	$$
	as $n\rightarrow \infty.$ 	This completes the proof of the existence of weak solutions for the regularized equation \eqref{eqn2.1}.
	%
	%
	%
	%
	%
	%
	%
	%
	%
	\subsection{Uniform bound estimates} We now show the uniform bound estimates of $f_{\e,\delta}$ with respect to $\e$ and $\delta$.  For this, we first recall the following technical lemma whose proof can be found in \cite[Lemma 2.5]{KMT14}.
	\begin{lemma}\cite{KMT14}		\label{lem4.4}  
		There exists a constant $C>0$, independent of $\e$ and $\delta$, such that
		\begin{align*}
			\sup_{y\in \T^N}\int_{{\T^N}}\theta_\epsilon(x-y){\frac{\rho_{f_{\e,\delta}}}{\theta_\epsilon*\rho_{f_{\e,\delta}}(x)}}\, dx
			\leq C
		\end{align*}
		for all nonnegative functions $\rho_{f_{\e,\delta}}\in L^1(\T^N)$.
	\end{lemma}

We then complete the proof of Proposition \ref{prop2.1} by providing the bound estimates in the following proposition. 
	\begin{proposition}\label{prop3.1}
		Let  $f_{\e,\delta}$ be the weak solution of \eqref{eqn2.1} constructed in Proposition \ref{prop2.1}. Then we have
\begin{align}\label{lp_fed}
\begin{aligned}
& \iint_{\mathbb{T}^N\times\mathbb{R}^N}f_{\epsilon,\delta}^p\,dvdx
			+p(p-1)\int_{0}^{t}\iint_{\mathbb{T}^N\times\mathbb{R}^N}f_{\epsilon,\delta}^{p-2}T_{f_{\e,\delta}}^{\e,\delta}|\nabla_vf_{\epsilon,\delta}|^2\,dvdxds \cr
&\quad			\leq e^{N(p-1)t}\iint_{\mathbb{T}^N\times\mathbb{R}^N}f_{0,\e}^p\,dvdx
		\end{aligned}
		\end{align}
		for any $p \ge 1$. In particular, we obtain
		\begin{align*}
			\sup_{0\leq t\leq T}\|f_{\e,\delta}(t)\|_{L^\infty}\leq \|f_{0,\e}\|_{L^\infty} e^{NT}.
		\end{align*}
Moreover, the third velocity-moment on $f_{\e,\delta}$ is uniformly bounded:
		\[
		\sup_{0 \leq t \leq T}\iint_{{\mathbb{T}^N}\times\mathbb{R}^N}|v|^3 f_{\e,\delta}(x,v,t)\, dvdx \leq C
		\]
		for some $C>0$ independent of $\epsilon$ and $\delta$.
	\end{proposition}
	
	\begin{remark} 	The bound estimate on $\|f_{\e,\delta}\|_{L^\infty}$ can be obtained by using the same argument as in the proof of Proposition \ref{prop2.2}. However, it does not give any uniform estimate of $\nabla_v f_{\e,\delta}$, which is required for the strong compactness via the velocity averaging, see Section \ref{ssec:compt} below. That is why we also provide the $L^p$ bound estimate of $f_{\e,\delta}$.
	\end{remark}
	\begin{proof}[Proof of Proposition \ref{prop3.1}]
	It follows from \eqref{eqn2.1} that 
	\begin{align*}
		&\frac{d}{dt}\iint_{\mathbb{T}^N\times\mathbb{R}^N}f_{\epsilon,\delta}^p\,dvdx\cr 
		&\quad =p\iint_{\mathbb{T}^N\times\mathbb{R}^N}f_{\epsilon,\delta}^{p-1}
		\Big( \nabla_v\cdot \left[ T_{f_{\e,\delta}}^{\e,\delta}\nabla_vf_{\epsilon,\delta}-\left( u_{f_{\e,\delta}}^\e-v\right) f_{\epsilon,\delta} \right] \Big) \,dvdx\\
		&\quad =-p(p-1)\iint_{\mathbb{T}^N\times\mathbb{R}^N}f_{\epsilon,\delta}^{p-2}\nabla_vf_{\epsilon,\delta}\cdot\Big( T_{f_{\e,\delta}}^{\e,\delta}\nabla_vf_{\epsilon,\delta}- ( u_{f_{\e,\delta}}^\e-v ) f_{\epsilon,\delta}\Big) \,dvdx
		\\
		&\quad =-p(p-1)\iint_{\mathbb{T}^N\times\mathbb{R}^N}f_{\epsilon,\delta}^{p-2}T_{f_{\e,\delta}}^{\e,\delta}|\nabla_vf_{\epsilon,\delta}|^2\,dvdx+(p-1)\iint_{\mathbb{T}^N\times\mathbb{R}^N} ( u_{f_{\e,\delta}}^\e-v )\cdot\nabla_v\{f_{\epsilon,\delta}^{p}\}\,dvdx\cr
		&\quad =-p(p-1)\iint_{\mathbb{T}^N\times\mathbb{R}^N}f_{\epsilon,\delta}^{p-2}T_{f_{\e,\delta}}^{\e,\delta}|\nabla_vf_{\epsilon,\delta}|^2\,dvdx+N(p-1)\iint_{\mathbb{T}^N\times\mathbb{R}^N}f_{\epsilon,\delta}^{p}\,dvdx.
	\end{align*}
We then apply Gr\"onwall's lemma to the above to conclude \eqref{lp_fed}. The $L^\infty$ bound estimate on $f_{\e,\delta}$ simply follows from \eqref{lp_fed}. Next, we show the bound estimate of the third velocity-moment on $f_{\e,\delta}$. Straightforward computation gives
	\begin{align*}
		{d\over dt}\iint_{{\mathbb{T}^N}\times\mathbb{R}^N}|v|^3 f_{\e,\delta}\, dvdx&= 3\iint_{{\mathbb{T}^N}\times\mathbb{R}^N} |v| v\cdot\left(u_{f_{\e,\delta}}^\epsilon-v\right)f_{\e,\delta} \,dvdx +3(N+1) \iint_{{\mathbb{T}^N}\times\mathbb{R}^N}|v|T_{f_{\e,\delta}}^{\e,\delta} f_{\e,\delta} \,dvdx\cr
		&=:\rom{1}+\rom{2}.
	\end{align*}
	Before estimating $\rom{1}$ and $\rom{2}$, we first observe
	\begin{align}\label{4.7}\begin{split}
			|u_{f_{\e,\delta}}^\epsilon|^3
			&=\left|\frac{\displaystyle \iint_{{\mathbb{T}^N}\times\mathbb{R}^N}\theta_\epsilon(x-y)\omega f_{\e,\delta}(y,\omega)\,dyd\omega}{\rho_{f_{\e,\delta}}*\theta_\epsilon+\epsilon\left(1+\left|\left(\rho_{f_{\e,\delta}} u_{f_{\e,\delta}}\right)*\theta_\epsilon\right|\right)}\right|^3\cr
			&\leq \left(\rho_{f_{\e,\delta}}*\theta_\epsilon\right)^{-3}\displaystyle \left(\iint_{{\mathbb{T}^N}\times\mathbb{R}^N}\theta_\epsilon(x-y) f_{\e,\delta}(y,\omega)\,dyd\omega\right)^{2}\iint_{{\mathbb{T}^N}\times\mathbb{R}^N}\theta_\epsilon(x-y)\left|\omega\right|^3 f_{\e,\delta}(y,\omega)\,dyd\omega\cr
			&=(\rho_{f_{\e,\delta}}*\theta_\epsilon)^{-1} \iint_{{\mathbb{T}^N}\times\mathbb{R}^N}\theta_\epsilon(x-y)\left|\omega\right|^3 f_{\e,\delta}(y,\omega)\,dyd\omega
	\end{split}\end{align}
and	
	\begin{align*}	
		 T_{f_{\epsilon,\delta}}^{\e,\delta} 
			&= \frac{\Phi_{f_{\epsilon,\delta}}^{\e,\delta}}{N\rho_{f_{\epsilon,\delta}}*\theta_\epsilon+\delta\big( 1+\Phi_{f_{\epsilon,\delta}}^{\e,\delta}\big) }+\frac{\delta^2}{N\rho_{f_{\epsilon,\delta}}*\theta_\epsilon+\delta\big( 1+\Phi_{f_{\epsilon,\delta}}^{\e,\delta}\big) } \cr
			&\leq \frac{\Phi_{f_{\epsilon,\delta}}^{\e,\delta}}{N\rho_{f_{\epsilon,\delta}}*\theta_\epsilon+\delta\big( 1+\Phi_{f_{\epsilon,\delta}}^{\e,\delta}\big) }
			+1 
	\end{align*}
due to $0<\delta<1$. On the other hand, by definition of $\Phi_{f}^{\e,\delta}$, we find
	\begin{align*}
		 \frac{\Phi_{f_{\epsilon,\delta}}^{\e,\delta}}{N\rho_{f_{\epsilon,\delta}}*\theta_\epsilon+\delta\big( 1+\Phi_{f_{\epsilon,\delta}}^{\e,\delta}\big) } 
		&\leq (\rho_{f_{\epsilon,\delta}}*\theta_\epsilon)^{-1}\displaystyle \int_{{\mathbb{T}^N}}\theta_\epsilon(x-y)\left(2\int_{{\mathbb{R}^N}} \left|\omega\right|^2f_{\epsilon,\delta}(y,\omega)\,d\omega\right)dy \cr
		&\leq2 (\rho_{f_{\epsilon,\delta}}*\theta_\epsilon)^{-1}\displaystyle \int_{{\mathbb{T}^N}}\theta_\epsilon(x-y)\left( \int_{{\mathbb{R}^N}}\left|\omega\right|^3 f_{\epsilon,\delta}(y,\omega)\,d\omega\right)^{\frac{2}{3}}\left(\rho_{f_{\epsilon,\delta}}(y)\right)^{\frac{1}{3}}\,dy\cr
		&\leq2 \left( \rho_{f_{\epsilon,\delta}}*\theta_\epsilon\right)^{-\frac{2}{3}}\displaystyle\left( \iint_{{\mathbb{T}^N}\times\mathbb{R}^N}\left|\omega\right|^3 f_{\epsilon,\delta}(y,\omega)\theta_\epsilon(x-y)\,d\omega dy\right)^{\frac{2}{3}} .
	\end{align*}
This implies
\bq	\label{4.8}
(T_{f_{\epsilon,\delta}}^{\e,\delta})^{\frac32} \leq C (\rho_{f_{\epsilon,\delta}}*\theta_\epsilon)^{-1}\displaystyle  \iint_{{\mathbb{T}^N}\times\mathbb{R}^N}\left|\omega\right|^3 f_{\epsilon,\delta}(y,\omega)\theta_\epsilon(x-y)\,d\omega dy     + C.
\eq
 $\bullet$ Estimate of $\rom{1}$: Applying Young's inequality, one finds
	\begin{align*}
		\rom{1}
		&\leq{\frac{1}{2}}\iint_{{\mathbb{T}^N}\times\mathbb{R}^N}3\left( {\frac{1}{3}} |u_{f_{\epsilon,\delta}}^\epsilon |^3 f_{\epsilon,\delta}+{\frac{2}{3}}|v|^3f_{\epsilon,\delta}\right)   dvdx
		-{\frac{3}{2}}\iint_{{\mathbb{T}^N}\times\mathbb{R}^N}|v|^3 f_{\epsilon,\delta}\, dvdx\cr
		&={\frac{1}{2}}\iint_{{\mathbb{T}^N}\times\mathbb{R}^N} |u_{f_{\epsilon,\delta}}^\epsilon |^3 f_{\epsilon,\delta} \,dvdx
		- {\frac{1}{2}}\iint_{{\mathbb{T}^N}\times\mathbb{R}^N}|v|^3 f_{\epsilon,\delta}\, dvdx,
	\end{align*}
	which, combined with  Lemma \ref{lem4.4} and \eqref{4.7}, gives
	\begin{align}\begin{split}
			\label{4.9}
			\int_{{\mathbb{T}^N}} |u_{f_{\epsilon,\delta}}^\epsilon |^3 \rho_{f_{\epsilon,\delta}}\, dx
			&\leq \iint_{{\mathbb{T}^N}\times\mathbb{R}^N}\left(\int_{{\mathbb{T}^N}}\theta_\epsilon(x-y)\frac{\rho_{f_{\epsilon,\delta}}(x)}{\theta_\epsilon*\rho_{f_{\epsilon,\delta}}(x)}dx\right)\left|\omega\right|^3 f_{\epsilon,\delta}(y,\omega) \,d\omega dy\cr
			&\leq C\iint_{{\mathbb{T}^N}\times\mathbb{R}^N}|v|^3 f_{\epsilon,\delta} \,dvdx 
	\end{split}\end{align}
	for some $C>0$ independent of $\e$ and $\delta$.
	\noindent\newline
	$\bullet$ Estimate of $\rom{2}$: It follows from H\"{o}lder's inequality and Young's inequality that
	\begin{align*}
		\rom{2}
		&\leq{\frac{3(N+1)}{2}}\int_{{\mathbb{T}^N}}T_{f_{\epsilon,\delta}}^{\e,\delta}\left( \int_{{\mathbb{R}^N}}|v|^3 f_{\epsilon,\delta}\, dv\right)^{\frac{1}{3}}\left( \int_{{\mathbb{R}^N}}f_{\epsilon,\delta}\, dv\right)^{\frac{1}{3}} dx\cr
		&\leq{\frac{N+1}{2}}\iint_{{\mathbb{T}^N\times\mathbb{R}^N}} |v|^3 f_{\epsilon,\delta}\, dvdx +3(N+1) \int_{{\mathbb{T}^N}}(T_{f_{\epsilon,\delta}}^{\e,\delta})^{\frac{3}{2}}\rho_{f_{\epsilon,\delta}} \,dx. 
	\end{align*}
	Similarly to \eqref{4.9}, by using Lemma \ref{lem4.4} and \eqref{4.8}, we get
	\begin{align}
		\label{4.10}
		3(N+1) \int_{{\mathbb{T}^N}}\left(T_{f_{\epsilon,\delta}}^\delta\right)^{\frac{3}{2}}\rho_{f_{\epsilon,\delta}}\, dx
		\leq C\iint_{{\mathbb{T}^N}\times\mathbb{R}^N}(1+|v|^3) f_{\epsilon,\delta} \,dvdx
	\end{align}
		for some $C>0$ independent of $\e$ and $\delta$.
		
	Combining all of the above estimates gives
	\[
	{d\over dt}\iint_{{\mathbb{T}^N}\times\mathbb{R}^N}|v|^3 f_{\e,\delta}\, dvdx \leq C\iint_{{\mathbb{T}^N}\times\mathbb{R}^N}(1 + |v|^3) f_{\e,\delta}\, dvdx,
	\]
	where $C>0$ is independent of $\e>0$. Finally, we apply Gr\"onwall's lemma to the above to conclude the desired result. 
	\end{proof}
	
	%
	%
	%
	%
	%
	%
	%
	%

	\section{Global existence of weak solutions}\label{sec:main}	
In this section, we provide the details of the proof for Theorem \ref{main result}.
	%
	%
	%
	%
	%
	%
	%
	%
\subsection{Strong compactness via velocity averaging}\label{ssec:compt}For the proof of Theorem \ref{main result}, we need some compactness results for the observable quantities. For this, we first provide some $L^p$ integrability of those quantities, which will be used to have the strong compactness via the celebrated velocity averaging lemma. 
	\begin{lemma}
		\label{lem5.1} Suppose that $h$ satisfies
		\begin{align*}
			\|h\|_{L^\infty(\T^N\times\R^N\times [0,T])}\leq M \quad \mbox{and} \quad
			\sup_{0\leq t\leq T}\iint_{{\mathbb{T}^N}\times\mathbb{R}^N}|v|^3 h\, dvdx\leq M
		\end{align*}
		for some $M>0$.		Then there exists a constant $C\equiv C(M)$ such that
		\begin{align*}
			&\|\rho_h\|_{L^\infty(0,T;L^p(\T^N))}\leq C \hspace{7.5mm} \mbox{for every} \quad p\in \lt[1,\frac{N+3}{N}\rt), \\
			&\|\rho_h u_h\|_{L^\infty(0,T;L^p(\T^N))}\leq C \quad \mbox{for every} \quad p\in \lt[1,\frac{N+3}{N+1}\rt), \quad \mbox{and}\\ 
			&\| N\rho_h T_h+\rho_h|u_h|^2\|_{L^\infty(0,T;L^p(\T^N))}\leq C \quad \mbox{for every} \quad p\in \lt[1,\frac{N+3}{N+2}\rt). 
		\end{align*}
	\end{lemma}
	\begin{proof}The proof is similar to that of \cite[Lemma 2.4]{KMT13}, but for the completeness of our work, we state details for readers' convenience.	For any $\varphi(v)$, we observe
	\[
	\lt|\intr \varphi(v) h\,dv \rt| \leq \intr (1 + |v|)^{\frac 3p} h^\frac1p \frac{h^{\frac{p-1}{p}} |\varphi(v)|  }{(1 + |v|)^{\frac 3p}}\,dv \leq \lt( \int (1+ |v|)^3  h\,dv\rt)^\frac1p \lt(\intr \frac{h |\varphi(v)|^{\frac p{p-1}}}{(1+ |v|)^\frac3{p-1}} \,dv\rt)^{\frac{p-1}{p}}.
	\]
	If $|\varphi(v)| \leq C|v|^m$, $m \ge 0$, and $\frac{3 - mp}{p-1} < N$, i.e. $p > \frac{N+3}{N+m}$, then
	\[
	\lt|\intr \varphi(v) h\,dv \rt| \leq C\|h\|_{L^\infty}^{\frac{p-1}{p}} \lt( \int (1+ |v|)^3  h\,dv\rt)^\frac1p.
	\]
	This implies
	\[
	\lt\|\intr \varphi(v) h\,dv \rt\|_{L^p}^p \leq CM
	\]
	for some $C>0$ independent of $M$. We now choose $\varphi(v) = 1, v, |v|^2$ to conclude the desired result. 
	\end{proof}
	
	We next recall the following velocity averaging lemma, whose proof can be found in \cite[Lemma 2.7]{KMT13}. 
\begin{lemma}\label{lem_velo}
Let   $f^m$ and $G^m$ satisfy
\[
\pa_t f^m + v \cdot \nabla_x f^m = \nabla_v^\ell G^m, \quad f^m|_{t=0} = f_0 \in L^p(\T^N \times \R^N)
\]
and $\{G^m\}$ be bounded in $L_{loc}^p(\T^N \times \R^N \times [0,T])$. Suppose that
\[
\sup_{m \in \bbn} \|f^m\|_{L^\infty(\T^N \times \R^N \times (0,T))} + \sup_{m \in \bbn}\| |v|^3 f^m\|_{L^\infty(0,T;L^1(\T^N \times \R^N))}  <\infty.
\]
Then, for any $\varphi(v)$ satisfying $|\varphi(v)| \ls 1+ |v|^2$, the sequence
\[ \left\{ \int_{\R^d} f^m \varphi \,dv \right\}_m \]
is relatively compact in $L^q(\T^N \times (0,T))$ for any $q \in  (1, \frac{N+3}{N+2}  )$.
\end{lemma}

\begin{remark} Strictily speaking, \cite[Lemma 2.7]{KMT13} takes into account the case in which $f^m$ has a finite second-velocity moment. However, by a simple modification of their proof, it can be readily extended to the case where $f^m$ has a finite third-velocity moment.
\end{remark}

In the rest of this section, we show that the regularized VFP equation \eqref{eqn2.1} satisfies the assumptions that appeared in Lemma \ref{lem_velo}.  
Let
\[
G_{\e,\delta}:=T_{f_{\e,\delta}}^{\e,\delta}\nabla_vf_{\epsilon,\delta}+\big(v-u_{f_{\epsilon,\delta}}^\epsilon\big)f_{\epsilon,\delta}.
\]
Then, it follows from \eqref{eqn2.1} that $f_{\epsilon,\delta}$ satisfies 
\[
\partial_tf_{\epsilon,\delta}+v\cdot\nabla_xf_{\epsilon,\delta} = \nabla_v \cdot G_{\e,\delta}.
\]
To employ the velocity averaging lemma and Lemma \ref{lem_velo}, we show that $G_{\e,\delta}$ is bounded in $L^q_{loc}(\T^N\times\R^N \times (0,T))$ for some $1<q<\infty$. Precisely, we claim that $G_{\e,\delta}$ is bounded in $L^\frac{6}{5}(\T^N\times\R^N \times (0,T))$. For this, we first obtain from \eqref{lp_fed} with $p=\frac{4}{3}$ that 
\bq\label{eqnq04}
\int_{0}^{T}\iint_{\mathbb{T}^N\times\mathbb{R}^N}|f_{\epsilon,\delta}|^{-\frac{2}{3}}T_{f_{\e,\delta}}^{\e,\delta}|\nabla_vf_{\epsilon,\delta}|^2\,dvdxdt < \infty
\eq
uniformly in both $\e$ and $\delta$. We then estimate
	\begin{align*}
		\|G_{\e,\delta}\|^{\frac 65}_{L^\frac65}
		&=\int_{0}^{T}\iint_{{\mathbb{T}^N}\times\R^N}\left|T_{f_{\e,\delta}}^{\e,\delta}\nabla_vf_{\epsilon,\delta}+ (v-u_{f_{\epsilon,\delta}}^\e )f_{\e}\right|^\frac65 dvdxdt\cr
		&\leq C\int_{0}^{T}\iint_{{\mathbb{T}^N}\times\R^N} (T_{f_{\e,\delta}}^{\e,\delta})^\frac65\left|\nabla_vf_{\epsilon,\delta}\right|^\frac65+|v|^\frac65 f_{\epsilon,\delta}^\frac65 + |u_{f_{\epsilon,\delta}}^\e|^\frac65 f_{\epsilon,\delta}^\frac65\,dvdxdt\cr
		&=:\rom{1}+\rom{2}+\rom{3},
	\end{align*}
where
 	\begin{align*}
		\rom{1}
		&=\int_{0}^{T}\iint_{{\mathbb{T}^N}\times\R^N}(T_{f_{\e,\delta}}^{\e,\delta})^\frac{3}{5}f_{\epsilon,\delta}^{\frac{2}{5}}f_{\epsilon,\delta}^{-\frac{2}{5}}(T_{f_{\e,\delta}}^{\e,\delta})^\frac{3}{5}\left|\nabla_vf_{\epsilon,\delta}\right|^\frac{6}{5}\,dvdxdt\cr
		&\leq \left(\int_{0}^{T}\iint_{{\mathbb{T}^N}\times\R^N}(T_{f_{\e,\delta}}^{\e,\delta})^\frac{3}{2}f_{\epsilon,\delta}\,dvdxdt\right)^{\frac{2}{5}}
		\left(\int_{0}^{T}\iint_{{\mathbb{T}^N}\times\R^N}|f_{\epsilon,\delta}|^{-\frac{2}{3}}T_{f_{\e,\delta}}^\delta|\nabla_vf_{\epsilon,\delta}|^2\,dvdxdt\right)^{\frac{3}{5}}.
	\end{align*}
	Then \eqref{4.10} and \eqref{eqnq04} imply $I$ is uniformly bounded in both $\e$ and $\delta$.
For $II$, we use Proposition \ref{prop3.1} and H\"older's inequality to get
\begin{align*}
	\rom{2}&=\int_{0}^{T}\iint_{{\mathbb{T}^N}\times\R^N}|v|^\frac{6}{5}f_{\epsilon,\delta}^\frac{3}{5}f_{\epsilon,\delta}^\frac{3}{5}\,dvdxdt\cr
	&\leq\left(\int_{0}^{T}\iint_{{\mathbb{T}^N}\times\R^N}|v|^2f_{\epsilon,\delta}\,dvdxdt\right)^\frac{3}{5}\left(\int_{0}^{T}\iint_{{\mathbb{T}^N}\times\R^N}f_{\epsilon,\delta}^\frac{3}{2}\,dvdxdt\right)^\frac{2}{5}\cr
	&\leq \left(\int_{0}^{T}\iint_{{\mathbb{T}^N}\times\R^N}|v|^2f_{\epsilon,\delta}\,dvdxdt\right)^\frac{3}{5}\left\|f_{\epsilon,\delta}\right\|_{L^{\infty}}^\frac{1}{5}\left(\int_{0}^{T}\iint_{{\mathbb{T}^N}\times\R^N}f_{\epsilon,\delta}\,dvdxdt\right)^\frac{2}{5}\cr
	&\leq C.
\end{align*}
In a similar way, using \eqref{4.9} and H\"older's inequality, we can estimate $\rom{3}$ as
\begin{align*}\begin{split}
		\rom{3}
		&\leq \left(\int_{0}^{T}\iint_{{\mathbb{T}^N}\times\R^N}|u_{f_{\epsilon,\delta}}^\e|^3f_{\epsilon,\delta}\,dvdxdt\right)^\frac{2}{5}\left(\int_{0}^{T}\iint_{{\mathbb{T}^N}\times\R^N}f_{\epsilon,\delta}^\frac{4}{3}\,dvdxdt\right)^\frac{3}{5}\cr
		&\leq \left(\int_{0}^{T}\iint_{{\mathbb{T}^N}\times\R^N}|u_{f_{\epsilon,\delta}}^\e|^3f_{\epsilon,\delta}\,dvdxdt\right)^\frac{2}{5}\left\|f_{\epsilon,\delta}\right\|_{L^{\infty}}^\frac{1}{5}\left(\int_{0}^{T}\iint_{{\mathbb{T}^N}\times\R^N}f_{\epsilon,\delta}\,dvdxdt\right)^\frac{3}{5}\cr
		&\leq C.
\end{split}\end{align*}
Combining all of the above estimates proves that $G_{\e,\delta}$ is bounded in $L^\frac{6}{5}(\T^N\times\R^N \times (0,T))$, and thus by Lemmas  \ref{lem5.1} and \ref{lem_velo}, we have
\[ 
\left\{ \int_{\R^N} f_{\e,\delta} \varphi(v) \,dv \right\}_{\e,\delta} \quad \mbox{with } \varphi(v) = 1, v, |v|^2
\]
is relatively compact in $L^q(\T^N \times (0,T))$ for any $q \in  (1, \frac{N+3}{N+2}  )$.

	%
	%
	%
	%
	%
	%
	%
	%
\subsection{Passing to the limit $\e \to 0$ (fixed $\delta > 0$)}

We now fix $\delta > 0$ and consider the limit $\e \to 0$. We recall our regularized observable quantities:
\begin{align*}
	u_{f_{\epsilon,\delta}}^\epsilon(x,t):=\frac{\left(\rho_{f_{\epsilon,\delta}} u_{f_{\epsilon,\delta}}\right)*\theta_\epsilon}{\rho_{f_{\epsilon,\delta}}*\theta_\epsilon+\epsilon\big(1+|(\rho_{f_{\epsilon,\delta}} u_{f_{\epsilon,\delta}})*\theta_\epsilon|^2\big)}
	\quad \mbox{and} \quad 
	T_{f_{\e,\delta}}^{\e,\delta}(x,t):=
	\frac{\Phi_{f_{\epsilon,\delta}}^\delta+\delta^2}{N\rho_{f_{\epsilon,\delta}}*\theta_\epsilon+\delta\big( 1+\Phi_{f_{\epsilon,\delta}}^\delta\big) }
\end{align*}
with
\begin{align*}
	\Phi_{f_{\epsilon,\delta}}^{\e,\delta}
	:=\left( N\rho_{f_{\epsilon,\delta}}T_{f_{\epsilon,\delta}}+\rho_{f_{\epsilon,\delta}}|u_{f_{\epsilon,\delta}}|^2\right) * \theta_\epsilon - \frac{\left| \rho_{f_{\epsilon,\delta}}u_{f_{\epsilon,\delta}}*\theta_\epsilon\right|^2}{\rho_{f_{\epsilon,\delta}}*\theta_\epsilon+\delta\big( 1+ | \rho_{f_{\epsilon,\delta}}u_{f_{\epsilon,\delta}}*\theta_\epsilon |^2\big)  }.
\end{align*}
Due to the uniform bound estimates in Proposition \ref{prop3.1},  we find $f_{\delta}\in  L^\infty(0,T; L^1\cap L^\infty(\mathbb{T}^N\times\mathbb{R}^N))$ such that $f_{\epsilon,\delta}$ weakly-$\star$ converges to $f_{\delta}$ in $  L^\infty(0,T; L^1\cap L^\infty(\mathbb{T}^N\times\mathbb{R}^N))$.
This, together with Lemmas \ref{lem5.1} and \ref{lem_velo}, implies  
\begin{align}
	\label{pconv}
	\rho_{f_{\epsilon,\delta}}\rightarrow \rho_{f_{\delta}}, \quad
	\rho_{f_{\epsilon,\delta}}u_{f_{\epsilon,\delta}}\rightarrow \rho_{f_{\delta}}u_{f_{\delta}}, \quad \mbox{and} \quad
	N\rho_{f_{\epsilon,\delta}}T_{f_{\epsilon,\delta}}+\rho_{f_{\epsilon,\delta}} |u_{f_{\epsilon,\delta}} |^2\rightarrow N\rho_{f_{\delta}}T_{f_{\delta}}+\rho_{f_{\delta}}\left|u_{f_{\delta}}\right|^2
\end{align}
in $L^p(\T^N\times(0,T))$ for every $p<\frac{N+3}{N+2} $ up to subsequence. Here $u_{f_{\delta}}$ and $T_{f_{\delta}}$ are defined as
\begin{equation*}
	u_{f_{\delta}}(x,t)
	:=\begin{cases}
		\frac{\rho_{f_{\delta}} u_{f_{\delta}}(x,t)}{\rho_{f_{\delta}}(x,t)} & \text{if $\rho_{f_{\delta}}(x,t) \neq 0$}\cr
		0 & \text{if $\rho_{f_{\delta}}(x,t) = 0$}
	\end{cases}
 \quad \mbox{and} \quad 
	T_{f_{\delta}}(x,t)
	:=\begin{cases}
		\frac{\rho_{f_{\delta}} T_{f_{\delta}}(x,t)}{\rho_{f_\delta}(x,t)} & \text{if $\rho_{f_{\delta}}(x,t) \neq 0$}\cr
		0 & \text{if $\rho_{f_{\delta}}(x,t) = 0$}
	\end{cases},
\end{equation*}
respectively.  In this subsection, our main goal is to show that the limit function ${f_{\delta}}$ satisfies the weak formulation of \eqref{eqn1.1} in the sense of Definition \ref{def1.1}, i.e., for any $\psi(x,v,t):=\phi(x,t)\varphi(v)$ with $\phi(x,t)\in \mc^1(\mathbb{T}^N\times[0,T])$ and $\varphi(v)\in \mc_c^2(\mathbb{R}^N)$ \mbox{with} $\psi(x,v,T)=0$
\begin{align*} 
	&-\iint_{\mathbb{T}^N\times\mathbb{R}^N}f_{0}\psi_0\, dvdx
	-\int_{0}^{T}\iint_{\mathbb{T}^N\times\mathbb{R}^N}{f_{\delta}}\left( \partial_t\psi+v\cdot\nabla_x\psi+(u_{f_{\delta}}-v)\cdot\nabla_v\psi \right) \, dvdxdt\\
	&\quad =\int_{0}^{T}\iint_{\mathbb{T}^N\times\mathbb{R}^N} T_{f_{\delta}}^\delta f_{\delta} \Delta_v\psi\, dvdxdt,
\end{align*}
where  $T_{f_{\delta}}^\delta$ is given by
\[
T_{f_\delta}^\delta=
	\frac{\Phi_{{f_\delta}}^\delta+\delta^2}{N\rho_{{f_\delta}}+\delta( 1+\Phi_{{f_\delta}}^\delta ) }.
\]
Here we consider the test function in $ \mc^1(\mathbb{T}^N\times[0,T])\otimes \mc_c^2(\mathbb{R}^N)$ instead, due to the density from the Stone-Weierstrass theorem. 
Note that by only using the weak compactness of $\{f_{\e,\delta}\}_{\e}$ in $L^\infty(\mathbb{T}^N\times\mathbb{R}^N\times(0,T))$, we can pass to the limit $\epsilon \rightarrow 0$ in the regularlized one \eqref{weak formulation} except $f_{\epsilon,\delta} u_{f_{\epsilon,\delta}}^\epsilon$ and $T_{f_{\epsilon,\delta}}^\delta f_{\epsilon,\delta}$ terms. Thus, for the desired result, it suffices to show that 
\[
f_{\epsilon,\delta} u_{f_{\epsilon,\delta}}^\epsilon \rightharpoonup f_\delta u_{f_\delta} \quad \mbox{and} \quad T_{f_{\epsilon,\delta}}^\delta f_{\epsilon,\delta} \rightharpoonup T_{f_\delta} f_\delta \quad
\quad\mbox{in}\quad L^{\infty}(0,T;L^p(\mathbb{T}^N\times\mathbb{R}^N))
\]
for $p<\frac{N+3}{N+2}$. For this, we first present an auxiliary lemma showing convergence results for convolutions.
\begin{lemma}\label{lem4.2}
The following convergences hold.
	\begin{align*}
		\rho_{f_{\epsilon,\delta}}*\theta_\epsilon\rightarrow\rho_{f_{\delta}},\quad 		\rho_{f_{\epsilon,\delta}} u_{f_{\epsilon,\delta}}*\theta_\epsilon\rightarrow\rho_{f_{\delta}} u_{f_{\delta}}, \quad \mbox{and} \quad  (N\rho_{f_{\epsilon,\delta}} T_{f_{\epsilon,\delta}}+\rho_{f_{\epsilon,\delta}}|u_{f_{\epsilon,\delta}}|^2)*\theta_\epsilon\rightarrow 	N\rho_{f_{\delta}} T_{f_{\delta}}+\rho_{f_{\delta}}|u_{f_{\delta}}|^2
	\end{align*}
	in $L^p(\T^N\times(0,T))$ for $p\in \left[ 1,\frac{N+3}{N+2}\right) $.
\end{lemma}
\begin{proof}
	Recall that for any function $h\in L^p(\T^N\times(0,T))$ with $1\leq p<\infty$, we have
	\begin{align*}
		\left\|h*\theta_\epsilon-h\right\|_{L^p(\T^N\times(0,T))}\rightarrow 0 \quad \mbox{as} \quad \epsilon\rightarrow 0.
	\end{align*}
	Using this, for any $p\in \left[ 1,\frac{N+3}{N+2}\right)$ we get
	\begin{align*}
		\left\|\rho_{f_{\epsilon,\delta}}*\theta_\epsilon-\rho_{f_{\delta}}\right\|_{L^p(\T^N\times(0,T))}
		&\leq\left\|\left( \rho_{f_{\epsilon,\delta}}-\rho_{f_{\delta}}\right) *\theta_\epsilon\right\|_{L^p(\T^N\times(0,T))}
		+\left\|\rho_{f_{\delta}}*\theta_\epsilon-\rho_{f_{\delta}}\right\|_{L^p(\T^N\times(0,T))}\cr
		&\leq\left\| \rho_{f_{\epsilon,\delta}}-\rho_{f_{\delta}}\right\|_{L^p(\T^N\times(0,T))}
		+\left\|\rho_{f_{\delta}}*\theta_\epsilon-\rho_{f_{\delta}}\right\|_{L^p(\T^N\times(0,T))}\cr
		&\rightarrow 0
	\end{align*}
	thanks to \eqref{pconv}. Similarly, we can also obtain the other convergences. 
\end{proof}
	%
	%
	%
	%
	%
	%
	%
	%
\subsubsection{Convergence  $f_{\epsilon,\delta} u_{f_{\epsilon,\delta}}^\epsilon \rightharpoonup f_\delta u_{f_\delta}$ }

First, let us denote 
\begin{align*}
	\rho_{f_{\epsilon,\delta}}^\varphi:=\int_{{\mathbb{R}^N}}f_{\epsilon,\delta} \nabla_v \varphi(v)\, dv,
\end{align*}
where $\varphi\in  \mc_c^1 (\mathbb{R}^N)$.
For any test function $\psi(x,v,t):=\phi(x,t)\varphi(v)$ with $\phi \in  \mc(\mathbb{T}^N\times[0,T])$, we have
\begin{align*}
	\int_0^T \iint_{{\mathbb{T}^N}\times\mathbb{R}^N}f_{\epsilon,\delta} u_{f_{\epsilon,\delta}}^\epsilon \cdot \nabla_v \psi\, dvdxdt
	&=\int_0^T\int_{{\mathbb{T}^N}}u_{f_{\epsilon,\delta}}^\epsilon \cdot \rho_{f_{\epsilon,\delta}}^\varphi\phi\, dxdt.
\end{align*}
Using the H\"older inequality, we get
\begin{align}\label{6.1}
	\begin{split}
		 \|u_{f_{\epsilon,\delta}}^\epsilon\rho_{f_{\epsilon,\delta}}^\varphi\|_{L^p}
		&\leq\|\nabla_v \varphi\|_{L^{\infty}}\left( \int_{{\mathbb{T}^N}}\left| u_{f_{\epsilon,\delta}}^\epsilon\rho_{f_{\epsilon,\delta}}\right|^p  dx\right)^{\frac{1}{p}} \leq \|\nabla_v \varphi\|_{L^{\infty}}\|\rho_{f_{\epsilon,\delta}}\|_{L^{\frac{p}{2-p}}}^{\frac{1}{2}}\|(\rho_{f_{\epsilon,\delta}})^{\frac{1}{2}}u_{f_{\epsilon,\delta}}^\epsilon\|_{L^2},
\end{split}	\end{align}
where we assumed $p\in  ( 1,{\frac{N+3}{N+2}} )$ so that  
\[
{\frac{p}{2-p}}\in\left( 1,{\frac{N+3}{N+1}}\right)\subset\left( 1,{\frac{N+3}{N}}\right). 
\]
Also, note that simple manipulation of \eqref{4.7} gives 
\begin{align*}
\begin{split}
		 |u_{f_{\epsilon,\delta}}^\epsilon |^2
		&\leq \left(\rho_{f_{\epsilon,\delta}}*\theta_\epsilon\right)^{-2} \left(\iint_{{\mathbb{T}^N}\times\mathbb{R}^N}\theta_\epsilon(x-y) f_{\epsilon,\delta}(y,\omega,t)\,dyd\omega\right)\iint_{{\mathbb{T}^N}\times\mathbb{R}^N}\theta_\epsilon(x-y)\left|\omega\right|^2 f_{\epsilon,\delta}(y,\omega,t) dyd\omega\cr
		&=(\rho_{f_{\epsilon,\delta}}*\theta_\epsilon)^{-1} \iint_{{\mathbb{T}^N}\times\mathbb{R}^N}\theta_\epsilon(x-y)\left|\omega\right|^2 f_{\epsilon,\delta}(y,\omega,t)\,dyd\omega,
\end{split}\end{align*}
and this provides
\begin{align*}\begin{split}
		\int_{{\mathbb{T}^N}} |u_{f_{\epsilon,\delta}}^\epsilon |^2\rho_{f_{\epsilon,\delta}}\, dx
		&\leq \iint_{{\mathbb{T}^N}\times\mathbb{R}^N}\left(\int_{{\mathbb{T}^N}}\theta_\epsilon(x-y)\frac{\rho_{f_{\epsilon,\delta}}}{\theta_\epsilon*\rho_{f_{\epsilon,\delta}}}dx\right)\left|\omega\right|^2f_{\epsilon,\delta} \,d\omega dy \leq C\iint_{{\mathbb{T}^N}\times\mathbb{R}^N}|v|^2f_{\epsilon,\delta}\, dvdx.
\end{split}\end{align*}
Thus, by Proposition \ref{prop3.1} and Lemma \ref{lem5.1}, the right-hand side of \eqref{6.1} is bounded. Hence, there exists a function $m_\delta \in L^{\infty}(0,T;L^p(\mathbb{T}^N))$ such that
\begin{align*}
	u_{f_{\epsilon,\delta}}^\epsilon\rho_{f_{\epsilon,\delta}}^\varphi \rightharpoonup 
	m_\delta
	\quad
	\mbox{in}
	\quad
	L^{\infty}(0,T;L^p(\mathbb{T}^N))
	\quad
	\mbox{for all}\quad
	p\in \left( 1,\frac{N+3}{N+2}\right)
\end{align*}
up to subsequence as $\e \to 0$. To identify the limit $m_\delta$, we consider a set $E_\delta$ defined as
\bq\label{e_delta}
	E_\delta:=\left\lbrace (x,t)\in \mathbb{T}^N\times [0,T] : \rho_{f_{\delta}}(x,t)\neq 0\right\rbrace.
\eq
Then, by \eqref{pconv}, \eqref{6.1} and above estimates, we observe
\begin{align*}
	\left\|u_{f_{\epsilon,\delta}}^\epsilon\rho_{f_{\epsilon,\delta}}^\varphi\right\|_{L^p(\T^N \times [0,T] \setminus E_\delta)}
	&\leq C \left\|\rho_{f_{\epsilon,\delta}}\right\|_{L^{\frac{p}{2-p}}(\T^N \times [0,T] \setminus E_\delta)}^{\frac{1}{2}}
	\rightarrow
	0
\end{align*}
as $\epsilon\rightarrow 0,$ where we assumed $p\in  ( 1,{\frac{2N+6}{2N+5}} )$ so that $\frac{p}{2-p}<\frac{N+3}{N+2}$.  
This shows that it is enough to obtain 
\begin{align*}
	m_\delta=u_{f_{\delta}}\rho_{f_{\delta}}^\varphi
	\quad \mbox{
		whenever}
	\quad
	\rho_{f_{\delta}}>0, 		\quad
	\mbox{	with }\quad
	\rho_{f_{\delta}}^\varphi=\int_{{\mathbb{R}^N}}f_{\delta}\varphi\,dv
	\quad \mbox{and}\quad
	\rho_{f_{\delta}} u_{f_{\delta}}=\int_{{\mathbb{R}^N}}vf_{\delta}\,dv.
\end{align*}
To prove this, for any $\gamma > 0$, let us consider a set
\[
	E_{\delta}^\gamma:=\left\lbrace (x,t)\in \mathbb{T}^N\times [0,T] :  \rho_{f_{\delta}}(x,t)>\gamma\right\rbrace .
\]
By Egorov's theorem and the compactness of $\rho_{f_{\epsilon,\delta}}$ and $\rho_{f_{\epsilon,\delta}}*\theta_\epsilon$, for any $\eta>0$, there exists a set $C_\eta\subset E_{\delta}^\gamma$ with $|E_{\delta}^\gamma\backslash C_\eta |<\eta$ on which both $\rho_{f_{\epsilon,\delta}}$ and $\rho_{f_{\epsilon,\delta}}*\theta_\epsilon$ uniformly converge to $\rho_{f_{\delta}}$ as $\epsilon\rightarrow 0$. Then for sufficiently small $\epsilon>0$, we have
\begin{align*}
	\rho_{f_{\epsilon,\delta}}*\theta_\epsilon>\frac{\gamma}{2}
	\quad \mbox{in}
	\quad
	C_\eta,
\end{align*}
and we pass to the limit on $C_\eta$ to get
\begin{align*}
	u_{f_{\epsilon,\delta}}^\epsilon\rho_{f_{\epsilon,\delta}}^\varphi
	=\frac{(\rho_{f_{\epsilon,\delta}} u_{f_{\epsilon,\delta}})*\theta_\epsilon}{\epsilon\big(1+\left|\left(\rho_{f_{\epsilon,\delta}} u_{f_{\epsilon,\delta}}\right)*\theta_\epsilon\right|^2\big)+\rho_{f_{\epsilon,\delta}}*\theta_\epsilon} \rho_{f_{\epsilon,\delta}}^\varphi
	\rightarrow
	m_\delta =u_{f_{\delta}}\rho_{f_{\delta}}^\varphi
	\quad
	\mbox{in} 
	\quad
	C_\eta,
\end{align*}
and this asserts
\begin{align*}
	m_\delta =u_{f_{\delta}}\rho_{f_{\delta}}^\varphi
	\quad
	\mbox{on}
	\quad
	E_\delta
\end{align*}
since $\eta>0$ and $\gamma>0$ are arbitrary. This implies
\begin{align*}
	\int_{0}^{T}\iint_{{\mathbb{T}^N}\times\mathbb{R}^N}f_{\epsilon,\delta} u_{f_{\epsilon,\delta}}^\epsilon \cdot \nabla_v \psi\, dvdxdt
	\rightarrow
	\int_{0}^{T}\int_{{\mathbb{T}^N}}u_{f_{\delta}}\rho_{f_{\delta}}^\varphi\phi\, dxdt
	=\int_{0}^{T}\iint_{{\mathbb{T}^N}\times\mathbb{R}^N}f_{\delta}u_{f_{\delta}} \cdot \nabla_v  \psi\, dvdxdt
\end{align*}
for any test functions of the form $\psi(x,v,t)=\phi(x,t)\varphi(v)$.

	%
	%
	%
	%
	%
	%
	%
	%
\subsubsection{Convergence  $T_{f_{\epsilon,\delta}}^{\e,\delta} f_{\epsilon,\delta} \rightharpoonup T_{f_\delta}^\delta f_\delta$ } 
Similarly as before, let us denote
\begin{align*}
	\tilde{\rho}_{f_{\epsilon,\delta}}^\varphi=\int_{{\mathbb{R}^N}}f_{\epsilon,\delta}\Delta_v\varphi(v)\, dv.
\end{align*}
For any test function $\psi(x,v,t):=\phi(x,t)\varphi(v)$ with $\phi \in  \mc(\mathbb{T}^N\times[0,T])$ and $\varphi\in \mc_c^{2}(\mathbb{R}^N)$, we have
\begin{align*}
	\iint_{{\mathbb{T}^N}\times\mathbb{R}^N}T_{f_{\epsilon,\delta}}^{\e,\delta}\Delta_vf_{\epsilon,\delta} \psi\, dvdx
	&=\int_{{\mathbb{T}^N}}T_{f_{\epsilon,\delta}}^{\e,\delta}\tilde{\rho}_{f_{\epsilon,\delta}}^\varphi\phi\, dx.
\end{align*}
Then, using the H\"older inequality yields
\begin{align}\label{6.6}\begin{split}
		\left\|T_{f_{\epsilon,\delta}}^{\e,\delta}\tilde{\rho}_{f_{\epsilon,\delta}}^\varphi\right\|_{L^p}
		&\leq\|\Delta_v\varphi\|_{L^{\infty}}\left( \int_{{\mathbb{T}^N}}\left| T_{f_{\epsilon,\delta}}^{\e,\delta}\rho_{f_{\epsilon,\delta}}\right|^p  dx\right)^{\frac{1}{p}}\cr
		&\leq  \|\Delta_v\varphi\|_{L^{\infty}}\|\rho_{f_{\epsilon,\delta}}\|_{L^{\frac{p}{3-2p}}}^{\frac{1}{3}}\|(\rho_{f_{\epsilon,\delta}})^{\frac{1}{2}}(T_{f_{\epsilon,\delta}}^{\e,\delta})^{\frac{3}{4}}\|_{L^2}^{\frac{4}{3}}.
\end{split}\end{align}
Note that
\[
{\frac{p}{3-2p}}\in\left( 1,{\frac{N+3}{N}}\right) \quad \mbox{for} \quad p\in \left( 1,{\frac{N+3}{N+2}}\right),
\]
and thus the right-hand side of \eqref{6.6} is bounded due to \eqref{4.10} and Lemma \ref{lem5.1}. This deduces that there exists a function $\tilde{m}_\delta \in L^{\infty}(0,T;L^p(\mathbb{T}^N))$ such that
\begin{align*}
	T_{f_{\epsilon,\delta}}^{\epsilon,\delta}\rho_{f_{\epsilon,\delta}}^\varphi \rightharpoonup 
	\tilde{m}_\delta \quad \mbox{in}	\quad	L^{\infty}(0,T;L^p(\mathbb{T}^N))
	\quad
	\mbox{for  all}\quad
	p\in \left( 1,{\frac{N+3}{N+2}}\right)
\end{align*}
up to subsequence as $\e \to 0$. On the other hand,
\bq\label{tde_con}
	\left\|T_{f_{\epsilon,\delta}}^{\epsilon,\delta}\rho_{f_{\epsilon,\delta}}^\varphi\right\|_{L^p(\T^N \times [0,T] \setminus E_\delta)}
	\leq C \left\|\rho_{f_{\epsilon,\delta}}\right\|_{L^{\frac{p}{3-2p}}(\T^N \times [0,T] \setminus E_\delta)}^{\frac{1}{3}}
	\rightarrow
	0
	\quad \mbox{ as}
	\quad
	\epsilon\rightarrow 0
\eq
due to \eqref{6.6}, where we assumed $p\in  ( 1,{\frac{3N+9}{3N+8}} )$ so that $\frac{p}{3-2p}<\frac{N+3}{N+2}$.   Thus it is sufficient to show
\begin{align*}
	\tilde{m}=T_{f_\delta}^\delta\rho_{f_\delta}^\varphi
	\quad \mbox{
		whenever}
	\quad
	\rho_{f_\delta}>0,
\end{align*}
where
\[
	\rho_{f_\delta}^\varphi=\int_{{\mathbb{R}^N}}{f_\delta}\Delta_v\varphi\, dv,
	\quad
	T_{f_\delta}^\delta=
	\frac{\Phi_{{f_\delta}}^\delta+\delta^2}{N\rho_{{f_\delta}}+\delta ( 1+\Phi_{{f_\delta}}^\delta ) },
	\]
	and
\[
	\Phi_{{f_\delta}}^\delta
	=N\rho_{f_\delta}T_{{f_\delta}}+\rho_{{f_\delta}}|u_{{f_\delta}}|^2 - \frac{\left| \rho_{{f_\delta}}u_{{f_\delta}}\right|^2}{\rho_{{f_\delta}}+\delta ( 1+\left| \rho_{{f_\delta}}u_{{f_\delta}}\right|^2 )  }.
\]
For this, we again use the set $E_\delta^\gamma$ and apply the Egorov theorem. Then, for any $\eta>0$, there exists a set $C_\eta\subset E_\delta^\gamma$ with $|E_\delta^\gamma\backslash C_\eta |<\eta$ on which both $\rho_{f_{\epsilon,\delta}}$ and $\rho_{f_{\epsilon,\delta}}*\theta_\epsilon$ uniformly converge to $\rho_{f_\delta}$ as $\epsilon\rightarrow 0$. Thus for sufficiently small $\epsilon>0$, we have
\begin{align*}
	\rho_{f_{\epsilon,\delta}}*\theta_\epsilon>{\frac{\gamma}{2}}
	\quad \mbox{in}
	\quad
	C_\eta,
\end{align*}
and we pass to the limit on $C_\eta$ to get
\begin{align}\label{novel2}\begin{split}
	\Phi_{f_{\epsilon,\delta}}^{\e,\delta}
	&=\left( N\rho_{f_{\epsilon,\delta}}T_{f_{\epsilon,\delta}}+\rho_{f_{\epsilon,\delta}}|u_{f_{\epsilon,\delta}}|^2\right) * \theta_\epsilon - \frac{\left| \rho_{f_{\epsilon,\delta}}u_{f_{\epsilon,\delta}}*\theta_\epsilon\right|^2}{\rho_{f_{\epsilon,\delta}}*\theta_\epsilon+\delta ( 1+ | \rho_{f_{\epsilon,\delta}}u_{f_{\epsilon,\delta}}*\theta_\epsilon |^2 )  }\cr
	&\rightarrow N\rho_{{f_\delta}}T_{{f_\delta}}+\rho_{{f_\delta}}|u_{{f_\delta}}|^2 - \frac{ | \rho_{{f_\delta}}u_{{f_\delta}} |^2}{\rho_{{f_\delta}}+\delta ( 1+\left| \rho_{{f_\delta}}u_{{f_\delta}}\right|^2 )  }\quad
	\mbox{in} \quad
	C_\eta,
\end{split}\end{align}
yielding
\begin{align*}
	T_{f_{\epsilon,\delta}}^{\epsilon,\delta}\rho_{f_{\epsilon,\delta}}^\varphi
	=\frac{\Phi_{f_{\epsilon,\delta}}^{\epsilon,\delta}+\delta^2}{N\rho_{f_{\epsilon,\delta}}*\theta_\epsilon+\delta\big( 1+\Phi_{f_{\epsilon,\delta}}^{\epsilon,\delta} \big)}\rho_{f_{\epsilon,\delta}}^\varphi
	\rightarrow
	\tilde{m}_\delta =T_{f_\delta}^\delta\rho_{f_\delta}^\varphi
	\quad
	\mbox{in} \quad
	C_\eta.
\end{align*}
This asserts
\begin{align*}
	\tilde{	m}_\delta=T_{f_\delta}^\delta\rho_{f_\delta}^\varphi
	\quad\mbox{ on}
	\quad
	E_\delta,
\end{align*}
since $\eta>0$ and $\gamma>0$ were arbitrary.
Hence we have
\begin{align*}
	\int_{0}^{T}\iint_{{\mathbb{T}^N}\times\mathbb{R}^N}T_{f_{\epsilon,\delta}}^\delta f_{\epsilon,\delta} \Delta_v\psi\,dvdxdt
	\rightarrow
	\int_{0}^{T}\int_{{\mathbb{T}^N}}T_{f_\delta}^\delta\rho_{f_\delta}^\varphi\phi\, dxdt=\int_{0}^{T}\iint_{{\mathbb{T}^N}\times\mathbb{R}^N}T_{f_\delta}^\delta  {f_\delta} \Delta_v\psi\, dvdxdt
\end{align*}
for any test functions of the form $\psi(x,v,t)=\phi(x,t)\varphi(v)$.

	%
	%
	%
	%
	%
	%
	%
	%
\subsection{Passing to the limit $\delta \to 0$ \& proof of the existence part of Theorem \ref{main result}}
 
 We now send $\delta \to 0$ and thereby conclude the proof of the existence part of Theorem \ref{main result}.  By using the uniform bound estimates and velocity averaging lemma again, there exists $f\in  L^\infty(0,T; L^1\cap L^\infty(\mathbb{T}^N\times\mathbb{R}^N))$ such that $f_{\delta}$ weakly-$\star$ converges to $f$ in $  L^\infty(0,T; L^1\cap L^\infty(\mathbb{T}^N\times\mathbb{R}^N))$. Thus, by Lemma \ref{lem5.1} and Lemma \ref{lem_velo}, we obtain 
\begin{align}
\label{pconv2}
\rho_{f_\delta}\rightarrow \rho_{f}, \quad
\rho_{f_\delta}u_{f_\delta}\rightarrow \rho_{f}u_{f}, \quad \mbox{and} \quad
N\rho_{f_\delta}T_{f_\delta}+\rho_{f_\delta}\left|u_{f_\delta}\right|^2\rightarrow N\rho_{f}T_{f}+\rho_{f}\left|u_{f}\right|^2
\end{align}
in $L^p(\T^N\times(0,T))$ for every $p<\frac{N+3}{N+2} $ up to subsequence as $\delta\rightarrow 0$. Here we recall 
	\begin{equation*}
		u_{f}(x,t)
		=\begin{cases}
			\frac{\rho_f u_f(x,t)}{\rho_f(x,t)} & \text{if $\rho_f(x,t) \neq 0$}\cr
			0 & \text{if $\rho_f(x,t) = 0$}
		\end{cases}
		\quad \mbox{and} \quad 
		T_{f}(x,t)
		=\begin{cases}
			\frac{\rho_f T_f(x,t)}{\rho_f(x,t)} & \text{if $\rho_f(x,t) \neq 0$}\cr
			0 & \text{if $\rho_f(x,t) = 0$}
		\end{cases}.
	\end{equation*}
Consider a set
\bq\label{de_ee}
	E:=\left\lbrace (x,t)\in \mathbb{T}^N\times (0,T)\; | \; \rho_{f}(x,t)\neq 0\right\rbrace.
\eq
Then, due to the convergence \eqref{pconv2}, we obtain
\[
 \iint_{({\mathbb{T}^N}\times(0,T)) \setminus E} \lt(\int_{\R^N }f_{\delta}u_{f_{\delta}} \cdot \nabla_v  \psi\, dv\rt)dxdt \to 0
\]
as $\delta\to 0$. On the other hand, we observe
\[
u_{f_\delta} = \frac{\rho_{f_\delta} u_{f_\delta}}{\rho_{f_\delta}} \to \frac{\rho_f u_f}{\rho_f} = u_f \quad \mbox{a.e. on $E$}
\]
and
\[
\int_{\R^N} f_\delta \nabla_v \psi\,dv \to \int_{\R^N} f \nabla_v \psi\,dv \quad \mbox{a.e. on $\T^N \times (0,T)$}
\]
as $\delta \to 0$. Moreover, we 	note that $\rho_{f_\delta}u_{f_\delta}\rightarrow \rho_{f}u_{f}$ in $ L^1(\T^N \times (0,T))$, and 
\begin{align*}
\left|\int_{{\mathbb{R}^N}}f_{\delta}u_{f_{\delta}} \cdot \nabla_v  \psi\, dv\right|
\leq \left|u_{f_{\delta}}\right||\nabla_v  \psi|\left|\int_{{\mathbb{R}^N}}f_{\delta}\,dv\right|
=|\nabla_v  \psi|\left|\rho_{f_\delta}u_{f_\delta}\right|
\leq C\left|\rho_{f_\delta}u_{f_\delta}\right|.
\end{align*} Thus, by employing the generalized Lebesgue-dominated convergence theorem, we deduce
\begin{align*}
	\int_{0}^{T}\iint_{{\mathbb{T}^N}\times\mathbb{R}^N}f_{\delta}u_{f_{\delta}} \cdot \nabla_v  \psi\, dvdxdt
	\rightarrow
	\int_{0}^{T}\iint_{{\mathbb{T}^N}\times\mathbb{R}^N}fu_{f} \cdot \nabla_v  \psi\, dvdxdt
\end{align*}
as $\delta\to 0$. We next show that
\bq\label{conv_tem}
	\int_{0}^{T}\iint_{{\mathbb{T}^N}\times\mathbb{R}^N}T_{f_\delta}^\delta  {f_\delta} \Delta_v\psi\, dvdxdt
	\rightarrow
	\int_{0}^{T}\iint_{{\mathbb{T}^N}\times\mathbb{R}^N}T_f  f \Delta_v\psi\, dvdxdt
\eq
as $\delta\to 0$ for any test functions of the form $\psi(x,v,t)=\phi(x,t)\varphi(v)$. For the proof, we use a similar argument as the above. Note that for $\delta < 1$
\[
\rho_{f_\delta}T_{f_\delta}^\delta = \frac{\rho_{f_\delta}(\Phi_{f_\delta}^\delta+\delta^2)}{N\rho_{f_\delta}+\delta\big( 1+\Phi_{f_\delta}^\delta\big) } \leq \frac 1N\Phi_{f_\delta}^\delta + \rho_{f_\delta}
\]
and
\begin{align*}
	\Phi_{{f_\delta}}^\delta
	=N\rho_{{f_\delta}}T_{{f_\delta}}+\rho_{{f_\delta}}|u_{{f_\delta}}|^2 - \frac{\left| \rho_{{f_\delta}}u_{{f_\delta}}\right|^2}{\rho_{{f_\delta}}+\delta ( 1+\left| \rho_{{f_\delta}}u_{{f_\delta}}\right|^2 )  }
	\leq N\rho_{{f_\delta}}T_{{f_\delta}}+\rho_{{f_\delta}}|u_{{f_\delta}}|^2.
\end{align*}
Using the above bound estimates, we find
\begin{align*}
	\left|T_{f_\delta}^\delta\int_{{\mathbb{R}^N}}{f_\delta}\Delta_v\psi\,dv\right|
	\leq T_{f_\delta}^\delta\int_{{\mathbb{R}^N}}{f_\delta}\left|\Delta_v\psi\right|\,dv
	\leq C\left(\frac{\Phi_{{f_\delta}}^\delta}{N\rho_{{f_\delta}}}+1\right)\rho_{{f_\delta}}
	\leq C\left(N\rho_{{f_\delta}}T_{{f_\delta}}+\rho_{{f_\delta}}|u_{{f_\delta}}|^2\right)+C\rho_{{f_\delta}},
\end{align*}
and this, together with \eqref{pconv2}, gives
\[
 \left|\iint_{({\mathbb{T}^N}\times(0,T)) \setminus E} \int_{\R^N }T_{f_\delta}^\delta  {f_\delta} \Delta_v\psi\,dvdxdt \right|\leq C  \iint_{({\mathbb{T}^N} \times(0,T)) \setminus E} \rho_{{f_\delta}}T_{{f_\delta}}+\rho_{{f_\delta}}|u_{{f_\delta}}|^2 + \rho_{f_\delta}\,dxdt \to 0
\]
as $\delta \to 0$. On the other hand, on $E$, we deduce
\begin{align*}
	\Phi_{{f_\delta}}^\delta
	=N\rho_{{f_\delta}}T_{{f_\delta}}+\rho_{{f_\delta}}|u_{{f_\delta}}|^2 - \frac{\left| \rho_{{f_\delta}}u_{{f_\delta}}\right|^2}{\rho_{{f_\delta}}+\delta\left( 1+\left| \rho_{{f_\delta}}u_{{f_\delta}}\right|^2\right)  }
	\rightarrow
	N\rho_{f}T_{f}+\rho_{f}|u_{f}|^2 - \frac{\left| \rho_{f}u_{f}\right|^2}{\rho_{f}  }
	=N\rho_{f}T_{f}
	\quad
	\mbox{a.e.},
\end{align*}
which implies
\begin{align*}
	T_{f_\delta}^\delta=
	\frac{\Phi_{{f_\delta}}^\delta+\delta^2}{N\rho_{{f_\delta}}+\delta\big( 1+\Phi_{{f_\delta}}^\delta\big) }
	\rightarrow 
	\frac{N\rho_{f}T_{f}}{N\rho_{f}}
	=T_f
	\quad
	\mbox{a.e.}
\end{align*}
Moreover, since
	\begin{align*}
	\rho_{{f_\delta}}
	\to
	\rho_f\quad\mbox{and}\quad 	
	N\rho_{{f_\delta}}T_{{f_\delta}}+\rho_{{f_\delta}}|u_{{f_\delta}}|^2
	\to
	N\rho_{f}T_{f}+\rho_{f}\left|u_{f}\right|^2	\quad\mbox{a.e.}\quad\mbox{and}\quad \mbox{in }L^1\big(\T^N\times(0,T)\big),
	\end{align*}
we can again use the generalized Lebesgue-dominated convergence theorem to conclude the desired convergence result \eqref{conv_tem}. Therefore, the limit function $f$ is the weak solution of the equation \eqref{eqn1.1} in the sense of Definition \ref{def1.1}.

	%
	%
	%
	%
	%
	%
	%
	%
	\subsection{Proof of Theorem \ref{main result}: conservation laws \& $H$-theorem}

	%
	%
	%
	%
	%
	%
	%
	%
\subsubsection{Conservation laws}\label{ssec:cons}
	Now we show that our weak solutions satisfy the following conservation laws:
	\begin{equation*}
	\iint_{\mathbb{T}^N\times\mathbb{R}^N} (1,v,|v|^2)f\,dvdx=\iint_{\mathbb{T}^N\times\mathbb{R}^N} (1,v,|v|^2)f_0\,dvdx.
	\end{equation*}
	Here, we only provide the details of the proof of energy conservation. The conservation of mass and momentum can be proved by using almost the same argument. 
	
	Let $R>1$ be given and consider a function $\varphi_R \in \mc^2(\R^d)$ satisfying 
	\[
	\chi_{B(0,R)} \leq \varphi_R \leq \chi_{B(0,2R)}, \quad |\nabla \varphi_R| \leq \frac {c_\varphi}R, \quad \mbox{and} \quad |\nabla^2 \varphi_R| \leq  \frac {c_\varphi}{R^2}
	\]
	for some $c_\varphi >0$ independent of $R$. 	We then take $|v|^2 \varphi_R(v)$ as a test function in the weak formulation for the regularized equation that appeared in \eqref{eqn2.1} to deduce 
\begin{align}\label{cons}
\begin{aligned}
& \iint_{\T^N \times \R^N} f_{\e,\delta}(t)|v|^2 \varphi_R(v)\,dvdx - \iint_{\T^N \times \R^N} f_{0,\e} |v|^2 \varphi_R(v)\,dvdx\cr
&\quad = \int_0^t \iint_{\T^N \times \R^N} f_{\e,\delta}(u^\epsilon_{f_{\e,\delta}} - v) \cdot \lt(2v \varphi_R(v) + |v|^2 \nabla \varphi_R(v) \rt)dvdxds\cr
&\qquad + \int_0^t \iint_{\T^N \times \R^N}  f_{\e,\delta} T^{\e,\delta}_{f_{\e,\delta}} \lt( 2N \varphi_R(v) + 4v \cdot\nabla \varphi_R(v) + |v|^2 \Delta \varphi_R(v)\rt) dvdxds.
\end{aligned}
\end{align}
Here, we first readily obtain 
\begin{align*}
&\int_0^t \iint_{\T^N \times \R^N}\lt( 2v  \cdot f_{\e,\delta}(u^{\e,\delta}_{f_\epsilon} - v)\varphi_R(v)   +  2N f_{\e,\delta} T^{\e,\delta}_{f_{\e,\delta}}  \varphi_R(v)\rt)dvdxds \cr
&\quad  \to \int_0^t \iint_{\T^N \times \R^N}\lt( 2v  \cdot f(u_f - v)\varphi_R(v)   +  2N f T_f  \varphi_R(v)\rt)dvdxds
\end{align*}
as $\epsilon, \delta \to 0$ due to weak-$\star$ convergences 
\[
f_{\e,\delta} u_{f_{\e,\delta}}^\epsilon \rightharpoonup fu_f \quad \mbox{and} \quad T_{f_{\e,\delta}}^{\e,\delta} f_{\e,\delta} \rightharpoonup T_f f \quad \mbox{in $L^\infty(0,T; L^1(\T^N \times \R^N))$}.
\]
This, together with noticing $|\varphi_R(v) - 1| \leq \frac{c_\varphi|v|}{R}$ by mean value theorem, further implies
\begin{align*}
&\lt|\int_0^t \iint_{\T^N \times \R^N}\lt( 2v  \cdot f(u_f - v)\varphi_R(v)   +  2N f T_f  \varphi_R(v)\rt)dvdxds\rt|\cr
&\quad \leq \frac{Cc_\varphi}R\int_0^t \iint_{\T^N \times \R^N} f |v|\lt( |u_f||v| + |v|^2 + T_f)\rt) dvdxds \leq \frac{Cc_\varphi}R\int_0^t \iint_{\T^N \times \R^N} f |v|^3 \,dvdxds \leq \frac{Cc_\varphi}R \to 0
\end{align*}
as $R \to \infty$, where we used 
\[
\int_{\R^N} v \cdot (u_f - v)f + Nf T_f\,dv= 0
\]
and the arguments in the proof of Proposition \ref{prop3.1}. For the other terms on the right-hand side of \eqref{cons}, similarly, we obtain 
\[
\lt|\int_0^t \iint_{\T^N \times \R^N} f_{\e,\delta}(u^\epsilon_{f_{\e,\delta}} - v) \cdot  |v|^2 \nabla \varphi_R(v) \,dvdxds\rt| \leq \frac{Cc_\varphi}R\int_0^t \iint_{\T^N \times \R^N} f_{\e,\delta}|v|^3 \,dvdxds \leq \frac{Cc_\varphi}R \to 0
\]
and
\begin{align*}
&\lt|\int_0^t \iint_{\T^N \times \R^N}  f_{\e,\delta} T^\epsilon_{f_{\e,\delta}} \lt(4v \cdot\nabla \varphi_R(v) + |v|^2 \Delta \varphi_R(v)\rt) dvdxds\rt| \cr
&\quad \leq \frac{Cc_\varphi}{R} \int_0^t \iint_{\T^N \times \R^N}  f_{\e,\delta} T^\epsilon_{f_{\e,\delta}} |v| \,dvdxds  \leq \frac{Cc_\varphi}R\int_0^t \iint_{\T^N \times \R^N} f_{\e,\delta}|v|^3 \,dvdxds \leq \frac{Cc_\varphi}R \to 0
\end{align*}
as $R \to \infty$. This shows that the right-hand side of \eqref{cons} converges to zero as $\epsilon,\delta \to 0$ and $R \to \infty$. Thus, it now remains to show that the left-hand side converges to 
\[
\iint_{\T^N \times \R^N} f(t) |v|^2\,dvdx - \iint_{\T^N \times \R^N} f_0 |v|^2\,dvdx
\]
 as $\epsilon,\delta \to 0$ and $R \to \infty$. Due to our choice of regularized initial data $f_{0,\e}$, we can easily find that 
\[
\iint_{\T^N \times \R^N} f_{0,\e} |v|^2 \varphi_R(v)\,dvdx \to \iint_{\T^N \times \R^N} f_0 |v|^2\,dvdx.
\]
We also use the weak-$\star$ convergence $f_{\epsilon,\delta} \rightharpoonup f$ in $L^\infty(\T^N \times \R^N \times (0,T))$ to obtain  	
\[
 \iint_{\T^N \times \R^N} f_{\e,\delta} |v|^2 \varphi_R(v)\,dvdx \to  \iint_{\T^N \times \R^N} f |v|^2 \varphi_R(v)\,dvdx
\]
as $\e, \delta \to 0$. Then by employing almost the same argument as the above, we have
\[
\lt|\iint_{\T^N \times \R^N} f |v|^2 \lt(\varphi_R(v)-  1\rt)\,dvdx\rt| \leq \frac{c_\varphi}{R}\iint_{\T^N \times \R^N} f |v|^3 \,dvdx \leq  \frac{C c_\varphi}{R} \to 0
\]
as $R \to \infty$, where $C>0$ is independent of $R$.

	%
	%
	%
	%
	%
	%
	%
	%
\subsubsection{Boltzmann's $H$-theorem}
We first notice that it is not clear how to derive the entropy inequality via the weak formulation in Definition \ref{def1.1} since our weak solution is not regular enough. For that reason, we recall the following regularized and linearized equation:
	\begin{align}
		\label{eqn_reg_r}
		\partial_tf^{n+1}_{\e,\delta}+v\cdot\nabla_xf^{n+1}_{\e,\delta}  =\nabla_v\cdot\lt(T_{f^n_{\e,\delta}}^{\epsilon, \delta}\nabla_vf^{n+1}_{\e,\delta}+\lt(v-u_{f^n_{\e,\delta}}^{\epsilon}\rt)f^{n+1}_{\e,\delta}\rt).
	\end{align}
Note that for fixed $\e, \delta>0$, $T_{f^n_{\e,\delta}}^{\e, \delta}$ is bounded from below by some positive constant due to Lemma \ref{lem_low}. 
Due to the strong regularities of $T_{f^n_{\e,\delta}}^{\e, \delta}$ and $u_{f^n_{\e,\delta}}^\epsilon$, and the initial data $f_{0,\e}$, we can  construct a sufficiently regular solution $f^{n+1}_{\e,\delta}$ to the equation \eqref{eqn_reg_r}. Then, by noticing that the equation \eqref{eqn_reg_r} can be written as
\[
\partial_tf^{n+1}_{\e,\delta}+v\cdot\nabla_xf^{n+1}_{\e,\delta}=\nabla_v\cdot\lt(T^{\e,\delta}_{f^n_{\e,\delta}}f^{n+1}_{\e,\delta}\nabla_v\ln\biggl(\frac{f^{n+1}_{\e,\delta}}{M^{\e,\delta}_{f^n_{\e,\delta}}}\biggl) \rt),
\]
where
 \[
 M^{\e,\delta}_{f}
	=\frac{\rho_f}{(2\pi T^{\e,\delta}_f)^\frac{N}{2}}\exp\lt(-\frac{|v-u^\epsilon_f|^2}{2T^{\e,\delta}_f}\rt),
\]
we estimate 
\begin{align*}	
\frac{d}{dt}\iint_{\T^N\times \R^N}f^{n+1}_{\e,\delta}\ln f^{n+1}_{\e,\delta} \,dvdx
	&=\iint_{\T^N\times \R^N}\nabla_v\cdot\left\{T^{\e,\delta}_{f^n_{\e,\delta}}f^{n+1}_{\e,\delta}\nabla_v\ln(f^{n+1}_{\e,\delta}/M^{\e,\delta}_{f^n_{\e,\delta}}) \right\}\ln f^{n+1}_{\e,\delta} \,dvdx\cr
	&=-\iint_{\T^N\times \R^N}T^{\e,\delta}_{f^n_{\e,\delta}}f^{n+1}_{\e,\delta}\left|\nabla_v\ln(f^{n+1}_{\e,\delta}/M^{\e,\delta}_{f^n_{\e,\delta}}) \right|^2 dvdx\cr
	&\quad-\iint_{\T^N\times \R^N}\nabla_v\cdot\left\{T^{\e,\delta}_{f^n_{\e,\delta}}\nabla_vf^{n+1}_{\e,\delta}+(v-u^\epsilon_{f^{n+1}_{\e,\delta}})f^{n+1}_{\e,\delta} \right\}\frac{|v-u^\epsilon_{f^n_{\e,\delta}}|^2}{2T^{\e,\delta}_{f^n_{\e,\delta}}} \,dvdx,
\end{align*} 
where the second term on the last identity is
\begin{align*}
	&-\iint_{\T^N\times \R^N}\nabla_v\cdot\left\{T^{\e,\delta}_{f^n_{\e,\delta}}\nabla_vf^{n+1}_{\e,\delta}+(v-u^\epsilon_{f^n_{\e,\delta}})f^{n+1}_{\e,\delta} \right\}\frac{|v-u^\epsilon_{f^n_{\e,\delta}}|^2}{2T^{\e,\delta}_{f^n_{\e,\delta}}} \,dvdx\cr
	&\quad =\iint_{{\mathbb{T}^N}\times\R^N}\left\{T^{\e,\delta}_{f^n_{\e,\delta}}\nabla_vf^{n+1}_{\e,\delta}+(v-u^\epsilon_{f^n_{\e,\delta}})f^{n+1}_{\e,\delta} \right\}\cdot\frac{v-u^\epsilon_{f^n_{\e,\delta}}}{T^{\e,\delta}_{f^n_{\e,\delta}}}\,dvdx\cr
	&\quad =-N\int_{{\mathbb{T}^N}}\rho_{f^{n+1}_{\e,\delta}}\,dx
	+\int_{{\mathbb{T}^N}}\frac{1}{T^{\e,\delta}_{f^n_{\e,\delta}}}\int_{{\mathbb{R}^N}}|v-u^\epsilon_{f^n_{\e,\delta}}|^2f^{n+1}_{\e,\delta}\,dvdx.
\end{align*}
Thus, we obtain
\begin{align}\label{h-thm-main0}
\begin{aligned}
&\iint_{\T^N\times \R^N}f^{n+1}_{\e,\delta}(t)\ln f^{n+1}_{\e,\delta}(t) \,dvdx  - \iint_{\T^N\times \R^N}f_{0,\e}\ln f_{0,\e} \,dvdx  \cr
&\quad \leq \int_{0}^{t}\int_{{\mathbb{T}^N}} \lt(-N\rho_{f^{n+1}_{\e,\delta}} + \frac{1}{T^{\e,\delta}_{f^n_{\e,\delta}}}\int_{{\mathbb{R}^N}}|v-u^\epsilon_{f^n_{\e,\delta}}|^2f^{n+1}_{\e,\delta}\,dv\rt) dxds.
\end{aligned}
\end{align}
Now we want to send $n \to \infty$ in \eqref{h-thm-main0}. Note that for fixed $\e,\delta > 0$, we have strong convergences (see Section \ref{ssec_cauchy}):
\[f^n_{\e,\delta} \to f_{\e,\delta} \quad \mbox{in } L^{\infty}(0,T;L_q^2(\mathbb{T}^N\times\mathbb{R}^N)), \qquad  \rho_{f^n_{\e,\delta}} \to \rho_{f_{\e,\delta}} \quad \mbox{in } L^{\infty}(0,T;L^2(\mathbb{T}^N)),
\] 
\[u_{f^n_{\e,\delta}}^{\e}\rightarrow u_{f_{\e,\delta}}^\e, \quad \mbox{and} \quad T^{\e,\delta}_{f^n_{\e,\delta}} \to T^{\e,\delta}_{f_{\e,\delta}} \quad \mbox{in } L^\infty(\T^N \times (0,T))
\]
and by Lemma \ref{lem_low}, $T^{\e,\delta}_{f^n_{\e,\delta}}$ is uniformly-in-$n$ bounded from below.
This yields that the right-hand side of \eqref{h-thm-main0} converges to
\[
\int_{0}^{t}\int_{{\mathbb{T}^N}} \lt(-N\rho_{f_{\e,\delta}} + \frac{1}{T^{\e,\delta}_{f_{\e,\delta}}}\int_{{\mathbb{R}^N}}|v-u^\epsilon_{f_{\e,\delta}}|^2f_{\e,\delta}\,dv\rt) dxds
\]
as $n \to \infty$. For the convergence of the left-hand side of \eqref{h-thm-main0}, we use the fact that $f \mapsto f \ln f$ is convex and $f^n_{\e,\delta} \to f_{\e,\delta}$ pointwise a.e. to apply the lower semi-continuity:
\[
 \intrr f_{\e,\delta} \ln f_{\e,\delta} \,dvdx \leq \liminf_{n \to \infty} \intrr f^{n+1}_{\e,\delta}\ln f^{n+1}_{\e,\delta} \,dvdx.
\]
Hence, we find
\begin{align}\label{h-thm-main}
\begin{aligned}
&\iint_{\T^N\times \R^N}f_{\e,\delta}(t)\ln f_{\e,\delta}(t) \,dvdx  - \iint_{\T^N\times \R^N}f_{0,\e}\ln f_{0,\e} \,dvdx  \cr
&\quad \leq \int_{0}^{t}\int_{{\mathbb{T}^N}} \lt(-N\rho_{f_{\e,\delta}} + \frac{1}{T^{\e,\delta}_{f_{\e,\delta}}}\int_{{\mathbb{R}^N}}|v-u^\epsilon_{f_{\e,\delta}}|^2f_{\e,\delta}\,dv\rt) dxds.
\end{aligned}
\end{align}
We then claim that the right-hand side of the above inequality goes to zero as $\e,\delta \to 0$.  For this, we first send $\e \to 0$ for fixed $\delta >0$. By direct calculation, we get
\begin{align*}
		\int_{{\mathbb{R}^N}}|v-u_{f_{\e,\delta}}^{\e}|^2f_{\e,\delta}\,dv
		=&\int_{{\mathbb{R}^N}}(|v|^2-2v\cdot u_{f_{\e,\delta}}^\e+|u_{f_{\e,\delta}}^\e|^2)f_{\e,\delta}\,dv\cr
		=&\int_{{\mathbb{R}^N}}|v|^2f_{\e,\delta}\,dv
		-2\int_{{\mathbb{R}^N}}v\cdot u_{f_{\e,\delta}}^\e f_{\e,\delta}\,dv
		+\int_{{\mathbb{R}^N}}|u_{f_{\e,\delta}}^\e|^2f_{\e,\delta}\,dv\cr
		=&N\rho_{f_{\e,\delta}}T_{f_{\e,\delta}}+\rho_{f_{\e,\delta}}|u_{f_{\e,\delta}}|^2
		-2\rho_{f_{\e,\delta}} u_{f_{\e,\delta}}\cdot u_{f_{\e,\delta}}^\e
		+|u_{f_{\e,\delta}}^\e|^2\rho_{f_{\e,\delta}}.
\end{align*}
Since we have considered a periodic spatial domain, \eqref{pconv} yields 
\begin{align}\label{dcal1}
\begin{aligned}
	&N\rho_{f_{\e,\delta}}T_{f_{\e,\delta}}+\rho_{f_{\e,\delta}}|u_{f_{\e,\delta}}|^2
	-2\rho_{f_{\e,\delta}} u_{f_{\e,\delta}}\cdot u_{f_{\e,\delta}}^\e
	+|u_{f_{\e,\delta}}^\e|^2\rho_{f_{\e,\delta}}\cr
	&\quad\rightarrow
	N\rho_{f_\delta}T_{f_\delta}+\rho_{f_\delta}|u_{f_\delta}|^2
	-2\rho_{f_\delta} u_{f_\delta}\cdot u_{f_\delta}
	+|u_{f_\delta}|^2\rho_{f_\delta}\cr
	&\quad=N\rho_{f_\delta}T_{f_\delta}
\end{aligned}
\end{align}
in $L^1(E_\delta)$ as $\e\rightarrow 0$, where $E_\delta$ is appeared in \eqref{e_delta}. We next notice that Lemma \ref{lem4.2} gives
\begin{align}
	\label{teda.e.}
	T_{f_{\e,\delta}}^{\e,\delta}
	=\frac{\Phi_{f_{\e,\delta}}^{\e,\delta}+\delta^2}{N\rho_{f_{\e,\delta}}*\theta_\epsilon+\delta\big( 1+\Phi_{f_{\e,\delta}}^{\e,\delta} \big)}
	\rightarrow
	\frac{\Phi_{f_\delta}^\delta+\delta^2}{N\rho_{f_\delta}+\delta\big( 1+\Phi_{f_\delta}^\delta \big)}
	=T_{f_\delta}^\delta
	\quad
	\mbox{a.e. on $E_\delta$},
\end{align}
as $\e \to 0$, where
\begin{align*}
	\Phi_{f_\delta}^\delta
	=N\rho_{f_\delta}T_{f_\delta}+\rho_{f_\delta}|u_{f_\delta}|^2 - \frac{\left| \rho_{f_\delta}u_{f_\delta}\right|^2}{\rho_{f_\delta}+\delta( 1+| \rho_{f_\delta}u_{f_\delta}|^2)  }.
\end{align*}
Since $\Phi_{f_\delta}^\delta \geq N\rho_{f_\delta}T_{f_\delta}$, we deduce 
\[
T_{f_\delta}^\delta \geq \frac{N\rho_{f_\delta}T_{f_\delta}+\delta^2}{N\rho_{f_\delta}+\delta( 1+ N\rho_{f_\delta}T_{f_\delta} )}.
\]
Moreover, by Lemma \ref{lemma2.1}, we get
\[
\rho_{f_\delta}\leq C\|f_\delta\|_{L^{\infty}}({T_{f_\delta}})^{\frac{N}{2}} \leq C({T_{f_\delta}})^{\frac{N}{2}}
\]
and subsequently, 
\[
\rho_{f_\delta}\leq C(\rho_{f_\delta}{T_{f_\delta}})^{\frac{N}{N+2}}\leq C(\rho_{f_\delta}{T_{f_\delta}} + 1)
\]
for some $C>0$ independent of $\delta$. This yields
\bq\label{lo_bdd_td}
T_{f_\delta}^\delta \geq \frac{N\rho_{f_\delta}T_{f_\delta}+\delta^2}{C( 1+ N\rho_{f_\delta}T_{f_\delta} )} \geq \frac\delta{C} 	\quad
	\mbox{a.e.}
\eq
due to $\delta < 1$. Then, we combine the above with \eqref{dcal1} and \eqref{teda.e.} to obtain
\[
\frac{1}{T^{\e,\delta}_{f_{\e,\delta}}}\int_{{\mathbb{R}^N}}|v-u^\epsilon_{f_{\e,\delta}}|^2f_{\e,\delta}\,dv \to \frac{N\rho_{f_\delta}T_{f_\delta}}{T_{f_\delta}^\delta} 
\]
in $L^1(E_\delta)$ as $\e \to 0$. Thus, we have
\begin{align*}
&\iint_{E_\delta} \biggl(-N \rho_{f_{\e,\delta}} + \frac{1}{T^{\e,\delta}_{f_{\e,\delta}}}\int_{{\mathbb{R}^N}}|v-u^\epsilon_{f_{\e,\delta}}|^2f_{\e,\delta}\,dv\biggl)dxds \to \iint_{E_\delta} \biggl(-N \rho_{f_\delta} + \frac{1}{T^\delta_{f_\delta}}\int_{{\mathbb{R}^N}}|v-u_{f_\delta}|^2f_\delta\,dv\biggl)dxds
\end{align*}
as $\e \to 0$.  On the other hand, we observe that  for $p\in [1,\frac{3N+9}{3N+8})$
\begin{align*}
	&\iint_{(\T^N \times (0,T))\setminus E_\delta}\rho_{f_{\e,\delta}}  |u^\e_{f_{\e,\delta}} |^2\,dxdt\cr
	&\quad \leq C\left(\iint_{(\T^N \times (0,T))\setminus E_\delta}\rho_{f_{\e,\delta}}^p |u^\e_{f_{\e,\delta}}|^{2p}\, dxdt\right)^{\frac{1}{p}}\cr
	&\quad \leq C \left(\iint_{(\T^N \times (0,T))\setminus E_\delta}\rho_{f_{\e,\delta}}^{\frac{p}{3-2p}} \,dxdt\right)^{\frac{3-2p}{3p}}
	\left(\iint_{(\T^N \times (0,T))\setminus E_\delta}\rho_{{f_{\e,\delta}}}|u^\e_{f_{\e,\delta}}|^{3}\,dxdt\right)^{\frac{2}{3}},
\end{align*}
where $\frac{p}{3-2p}<\frac{N+3}{N+2}$. Thus, it follows from Proposition \ref{prop3.1}, \eqref{4.9}, and \eqref{pconv} that
\begin{align*}
	\iint_{(\T^N \times (0,T))\setminus E_\delta}\rho_{f_{\e,\delta}} |u^\e_{f_{\e,\delta}}|^2\,dxdt
	\leq C\left\|\rho_{{f_{\e,\delta}}}\right\|_{L^{\frac{p}{3-2p}}((\T^N \times (0,T))\setminus E_\delta)}^{\frac{1}{3}} \to 0 
\end{align*}
as $\e\to 0$. Similarly, we obtain
\[
\iint_{(\T^N \times (0,T))\setminus E_\delta}\rho_{f_{\e,\delta}} |u_{f_{\e,\delta}}|^2\,dxdt \to 0
\]
and
\[
\iint_{(\T^N \times (0,T))\setminus E_\delta} \rho_{f_{\e,\delta}}T^{\epsilon,\delta}_{f_{\e,\delta}}\,dxdt \to 0
\]
as $\e \to 0$, see also \eqref{tde_con}. Thus, we obtain
\begin{align*}
&\iint_{(\T^N \times (0,T))\setminus E_\delta} N\rho_{f_{\e,\delta}}T^{\epsilon,\delta}_{f_{\e,\delta}}+\rho_{f_{\e,\delta}}|u_{f_{\e,\delta}}-u_{f_{\e,\delta}}^\e|^2 \,dxdt\cr
	&\quad \leq  2\iint_{(\T^N \times (0,T))\setminus E_\delta} N\rho_{f_{\e,\delta}}T^{\epsilon,\delta}_{f_{\e,\delta}}+\rho_{f_{\e,\delta}}|u_{f_{\e,\delta}}|^2
		+|u_{f_{\e,\delta}}^\e|^2\rho_{f_{\e,\delta}} \,dxdt\cr
		&\quad \to 0
\end{align*}
as $\e \to 0$. This together with \eqref{lo_bdd_td} yields
\[
\iint_{(\T^N \times (0,T))\setminus E_\delta} \biggl(-N \rho_{f_{\e,\delta}} + \frac{1}{T^{\e,\delta}_{f_{\e,\delta}}}\int_{{\mathbb{R}^N}}|v-u^\epsilon_{f_{\e,\delta}}|^2f_{\e,\delta}\,dv\biggl)dxds \to 0
\]
and subsequently,
\[
 \int_{0}^{t}\int_{{\mathbb{T}^N}} \biggl(-N\rho_{f_{\e,\delta}} + \frac{1}{T^{\e,\delta}_{f_{\e,\delta}}}\int_{{\mathbb{R}^N}}|v-u^\epsilon_{f_{\e,\delta}}|^2f_{\e,\delta}\,dv\biggl) dxds \to  \int_{0}^{t}\int_{{\mathbb{T}^N}} \biggl(-N\rho_{f_\delta} + \frac{1}{T^\delta_{f_\delta}}\int_{{\mathbb{R}^N}}|v-u_{f_\delta}|^2f_\delta\,dv\biggl) dxds
\]
as $\e \to 0$. We finally send $\delta \to 0$. We again use
\begin{align*}
	\Phi_{f_\delta}^\delta
	=N\rho_{f_\delta}T_{f_\delta}+\rho_{f_\delta}|u_{f_\delta}|^2 - \frac{\left| \rho_{f_\delta}u_{f_\delta}\right|^2}{\rho_{f_\delta}+\delta ( 1+| \rho_{f_\delta}u_{f_\delta}|^2 )  }
	\geq N\rho_{f_\delta}T_{f_\delta}
\end{align*}
to get 
\begin{align*}
	\frac{N\rho_{f_\delta}T_{f_\delta}}{T^\delta_{{f_\delta}}}
	\leq\frac{N\rho_{{f_\delta}}+\delta\big( 1+N \rho_{{f_\delta}} T_{{f_\delta}} \big)}{N \rho_{{f_\delta}} T_{{f_\delta}}+\delta^2}N\rho_{f_\delta}T_{f_\delta}
	\leq N\rho_{{f_\delta}}+\delta\big( 1+N \rho_{{f_\delta}} T_{{f_\delta}} \big).
\end{align*}
Observe that 
\begin{align*}
	&N\rho_{{f_\delta}}T_{f_\delta}
	=\int_{{\mathbb{R}^N}}|v|^2f_\delta\,dv-\frac{\left|\rho_{f_\delta}u_{f_\delta}\right|^2}{\rho_{f_\delta}}
\to
	\int_{{\mathbb{R}^N}}|v|^2f\,dv-\frac{\left|\rho_{f}u_{f}\right|^2}{\rho_{f}}
	=N\rho_{{f}}T_{f}
	\quad
	\mbox{a.e.}
\end{align*}
and
\begin{align*}
	T_{f_\delta}^\delta=
	\frac{\Phi_{{f_\delta}}^\delta+\delta^2}{N\rho_{{f_\delta}}+\delta\big( 1+\Phi_{{f_\delta}}^\delta\big) }
	\rightarrow 
	\frac{N\rho_{f}T_{f}}{N\rho_{f}}
	=T_f
	\quad
	\mbox{a.e.}
\end{align*}
on a set $E$ defined as in \eqref{de_ee}. This, together with using the facts that $\displaystyle\rho_{f}\in L^1\big(\T^N\times(0,T)\big)$ and
\begin{align*}
	N\rho_{{f_\delta}}+\delta\big( 1+N \rho_{{f_\delta}} T_{{f_\delta}} \big)
	\to
	N\rho_f
	\quad
	\mbox{a.e.}
\end{align*}
as $\delta\rightarrow 0$, enables us to use the generalized Lebesgue-dominated convergence theorem to get
\begin{align*}
 \iint_{E} \lt(-N\rho_{f_\delta} + \frac{1}{T^\delta_{f_\delta}}\int_{{\mathbb{R}^N}}|v-u_{f_\delta}|^2f_\delta\,dv\rt) dxds
	&\rightarrow
	 \iint_{E} \lt(-N \rho_{f} + \frac{N\rho_fT_f}{T_{f}}\rt)dxds = 0
\end{align*}
as $\delta\rightarrow 0$. On the other hand, for the estimate on $(\T^N \times (0,T))\setminus E$, we find
\begin{align*}
	0
	\leq\iint_{(\T^N \times (0,T)) \setminus E}\frac{N\rho_{f_\delta}T_{f_\delta}}{T^\delta_{{f_\delta}}}\,dxdt
	&=\iint_{(\T^N \times (0,T)) \setminus E}\frac{N\rho_{f_\delta}+\delta(1+\Phi_{f_\delta}^\delta)}{\Phi_{f_\delta}^\delta+\delta^2}N\rho_{f_\delta}T_{f_\delta}\,dxdt\cr
	&\leq \iint_{(\T^N \times (0,T)) \setminus E}\frac{N\rho_{f_\delta}+\delta(1+\Phi_{f_\delta}^\delta)}{N\rho_{f_\delta}T_{f_\delta}}N\rho_{f_\delta}T_{f_\delta}\,dxdt\cr
	&=\iint_{(\T^N \times (0,T)) \setminus E}N\rho_{f_\delta}+\delta(1+\Phi_{f_\delta}^\delta)\,dxdt.
\end{align*}
Since
\begin{align*}
	\Phi_{f_\delta}^\delta =N\rho_{{f_\delta}}T_{{f_\delta}}+\rho_{{f_\delta}}|u_{{f_\delta}}|^2 - \frac{\left| \rho_{{f_\delta}}u_{{f_\delta}}\right|^2}{\rho_{{f_\delta}}+\delta( 1+| \rho_{{f_\delta}}u_{{f_\delta}}|^2)  }
	\leq N\rho_{{f_\delta}}T_{{f_\delta}}+\rho_{{f_\delta}}|u_{{f_\delta}}|^2,
\end{align*}
we get
\begin{align*}
	\iint_{(\T^N \times (0,T)) \setminus E}N\rho_{f_\delta}+\delta(1+\Phi_{f_\delta}^\delta)\,dxdt
	\rightarrow
	\iint_{(\T^N \times (0,T)) \setminus E}N\rho_f\,dxdt
	=0.
\end{align*}
Thus, we have
\begin{align*}
	\iint_{(\T^N \times (0,T)) \setminus E}\frac{N\rho_{f_\delta}T_{f_\delta}}{T^\delta_{{f_\delta}}}\,dxdt
	\rightarrow0
\end{align*}
as $\delta\to 0$, and this yields that the right-hand side of \eqref{h-thm-main} coverges to zero as $\e,\delta \to 0$, i.e., 
\bq\label{entropy}
\lim_{\delta \to 0} \lim_{\e \to 0}\iint_{\T^N\times \R^N}f_{\e,\delta}\ln f_{\e,\delta} \,dvdx \leq   \lim_{\e \to 0} \iint_{\T^N\times \R^N}f_{0,\e}\ln f_{0,\e} \,dvdx .
\eq
Moreover, we notice that $f \mapsto f \ln f$ is convex and
\[
f_{\e,\delta} \rightharpoonup f \quad \mbox{in $L^p(\T^N \times \R^N \times (0,T))$ for any $p < \infty$}
\]
thus by the weak lower semi-continuity, we deduce
\[
 \intrr f \ln f \,dvdx \leq \liminf_{\e,\delta \to 0} \intrr f_{\e,\delta}\ln f_{\e,\delta} \,dvdx.
\]
This, together with \eqref{entropy}, gives
\[
 \intrr f \ln f \,dvdx \leq   \lim_{\e \to 0} \iint_{\T^N\times \R^N}f_{0,\e}\ln f_{0,\e} \,dvdx.
\]
We finally claim that 
\bq\label{fin_cla}
  \lim_{\e \to 0} \iint_{\T^N\times \R^N}f_{0,\e}\ln f_{0,\e} \,dvdx = \iint_{\T^N\times \R^N}f_0 \ln f_0 \,dvdx.
\eq
For this, we notice that for any $s,\sigma\geq 0$,
\begin{align}
	\label{nlog}
	-s(\ln s)\chi_{\left\{0\leq s\leq 1\right\}}
	&=-s(\ln s)\chi_{\left\{e^{-\sigma}\leq s\leq 1\right\}}
	-s(\ln s)\chi_{\left\{e^{-\sigma}\geq s\right\}} \leq s\sigma+C\sqrt{s}\chi_{\left\{e^{-\sigma}\geq s\right\}} \leq s\sigma+Ce^{-\frac{\sigma}{2}}
\end{align}
for some constant $C>0$ independent of $s$ and $\sigma$. Thus, by choosing $s=f_{0,\e}$ and $\sigma=|v|^2$ in \eqref{nlog}, we get
\begin{align*}
f_{0,\e}|\ln f_{0,\e}| & \leq  f_{0,\e}\ln f_{0,\e}\chi_{\left\{f_{0,\e}\geq 1\right\}}  -  f_{0,\e}\ln f_{0,\e}\chi_{\left\{0\leq f_{0,\e}\leq 1\right\}}  \cr
& \leq  f_{0,\e}^2  +   f_{0,\e}|v|^2+Ce^{-\frac{|v|^2}{2}} =: g_\e.
\end{align*}
Note that 
\[
f_{0,\e}\ln f_{0,\e} \to f_0 \ln f_0 \quad \mbox{and}\quad g_\e \to f_0^2 + f_0|v|^2 +Ce^{-\frac{|v|^2}{2}} \quad \mbox{pointwise a.e.}
\]
as $\e  \to 0$, and
\[
 \lim_{\e  \to 0}  \intrr g_\e\,dvdx = \intrr f_0^2 + f_0|v|^2 +Ce^{-\frac{|v|^2}{2}}\,dvdx.
\]
Hence, by the generalized Lebesgue-dominated convergence theorem, we obtain \eqref{fin_cla} and conclude the desired entropy inequality:
\[
 \intrr f \ln f \,dvdx \leq  \intrr f_0 \ln f_0 \,dvdx.
\]

\appendix

\section{Proof of Lemma \ref{lq2}}\label{app_a}
In this appendix, we provide the rest of the proof of Lemma \ref{lq2}, i.e. 
\begin{align}
\begin{split}
\label{derv}
&\frac{d}{dt}\left(\|\nabla_{x}f^{n+1}(t)\|_{L_q^2}^2
+\|\nabla_{v}f^{n+1}(t)\|_{L_q^2}^2\right)
+c_{\epsilon,\delta}
\|\nabla_{x}\nabla_{v}f^{n+1}\|_{L_q^2}^2 \cr
&\quad \leq C_{\e,\delta,q,N}
\left(\|\nabla_{x}f^{n+1}(t)\|_{L_q^2}^2
+\|\nabla_{v}f^{n+1}(t)\|_{L_q^2}^2\right).
\end{split}
\end{align}
To this end, we follow the same line as in the proof of Lemma \ref{lq2}. Observe that
\begin{align*}
	&{d\over dt}\iint_{\mathbb{T}^N\times\mathbb{R}^N}(1+|v|^q)|\nabla_{x}f^{n+1}|^2~dvdx\cr
	&\quad=2\iint_{\mathbb{T}^N\times\mathbb{R}^N}(1+|v|^q)\nabla_{x}f^{n+1}\cdot\Big(\nabla_{x}\Big[ T_{f^n}^{\epsilon,\delta}\Delta_{v}f^{n+1}+Nf^{n+1}-\left( u_{f^n}^{\epsilon}-v\right)\cdot \nabla_{v}f_{\epsilon}^{n+1}\Big]\Big)\cr
	&\quad=2\iint_{\mathbb{T}^N\times\mathbb{R}^N}(1+|v|^q)\nabla_{x}f^{n+1}\cdot\Big[T_{f^n}^{\epsilon,\delta}\nabla_{x}\Delta_{v}f^{n+1}+\nabla_{x}T_{f^n}^{\epsilon,\delta}\Delta_{v}f^{n+1}\Big]~dvdx\cr
	&\qquad-2\iint_{\mathbb{T}^N\times\mathbb{R}^N}(1+|v|^q)\nabla_{x}f^{n+1}\cdot\Big[\nabla_{x} u_{f^n}^{\epsilon} \nabla_{v}f^{n+1}\Big]~dvdx\cr
	&\qquad+2\iint_{\mathbb{T}^N\times\mathbb{R}^N}(1+|v|^q)\nabla_{x}f^{n+1}\cdot\Big[N\nabla_{x}f^{n+1}-(u_{f^n}^{\epsilon}-v)\cdot\nabla_{x}\nabla_{v}f^{n+1} \Big]~dvdx\cr
	&\quad=: \rom{1}+\rom{2}+\rom{3}+\rom{4}+\rom{5}
\end{align*}
and
\begin{align*}
	&{d\over dt}\iint_{\mathbb{T}^N\times\mathbb{R}^N}(1+|v|^q)|\nabla_{v}f^{n+1}|^2~dvdx\cr
	&\quad=2\iint_{\mathbb{T}^N\times\mathbb{R}^N}(1+|v|^q)\nabla_{v}f^{n+1}\cdot\Big(\nabla_{v}\Big[ T_{f^n}^{\epsilon,\delta}\Delta_{v}f^{n+1}+Nf^{n+1}\Big]-\nabla_{x}f^{n+1}\Big)~dvdx\cr
	&\qquad-2\iint_{\mathbb{T}^N\times\mathbb{R}^N}(1+|v|^q)\nabla_{v}f^{n+1}\cdot\nabla_{v}\Big[(u_{f^n}^{\epsilon}-v)\cdot\nabla_{v}f^{n+1}\Big]~dvdx\cr
	&\quad=: \rom{6} + \rom{7} + \rom{8} + \rom{9}.
\end{align*}
	$\bullet$ Estimate of $\rom{1}$:
	It follows from  integration by parts and Lemma \ref{lem_low} that
	\begin{align}\label{I}\begin{split}
		 \rom{1} 
	&=-\iint_{\mathbb{T}^N\times\mathbb{R}^N}qv|v|^{q-2}\cdot\nabla_{v}|\nabla_{x}f^{n+1}|^2T_{f^n}^{\epsilon,\delta}~dvdx-2\iint_{\mathbb{T}^N\times\mathbb{R}^N}(1+|v|^q)T_{f^n}^{\epsilon,\delta}|\nabla_{x}\nabla_{v}f^{n+1}|^2~dvdx\cr
		&=\iint_{\mathbb{T}^N\times\mathbb{R}^N}q(N+q-2)|v|^{q-2}|\nabla_{x}f^{n+1}|^2T_{f^n}^{\epsilon,\delta}~dvdx-2\iint_{\mathbb{T}^N\times\mathbb{R}^N}(1+|v|^q)T_{f^n}^{\epsilon,\delta}|\nabla_{x}\nabla_{v}f^{n+1}|^2~dvdx\cr
		&\le C_{q,N}\|T_{f^n}^{\epsilon,\delta}\|_{L^{\infty}}\iint_{\mathbb{T}^N\times\mathbb{R}^N}(1+|v|^q)|\nabla_{x}f^{n+1}|^2~dvdx-c_{\epsilon,\delta}\iint_{{\mathbb{R}^N}\times\mathbb{R}^N}(1+|v|^q)|\nabla_{x}\nabla_{v}f^{n+1}|^2~dvdx\cr
		&\leq C_{\epsilon,\delta,q,N}\|\nabla_{x}f^{n+1}\|_{L_{q}^2}^2
		-c_{\epsilon,\delta}\|\nabla_{x}\nabla_{v}f^{n+1}\|_{L_{q}^2}^2.
\end{split}	\end{align}
	$\bullet$ Estimate of $\rom{2}$:
	We use integration by parts, H\"older's inequality and Young's inequality to obtain
	\begin{flalign*}
		\rom{2}
		&=-2\iint_{\mathbb{T}^N\times\mathbb{R}^N}qv|v|^{q-2}\left( \nabla_{x}f^{n+1}\cdot\nabla_{x}T_{f^n}^{\epsilon,\delta}\right)\cdot\nabla_{v}f^{n+1}~dvdx\cr
		&\quad -2\iint_{\mathbb{T}^N\times\mathbb{R}^N}(1+|v|^q)\left( \nabla_{v}\nabla_{x}f^{n+1}\nabla_{x}T_{f^n}^{\epsilon,\delta}\right)\cdot\nabla_{v}f_\epsilon^{n+1}~dvdx\cr
		&\leq C_q\|\nabla_{x}T_{f^n}^{\epsilon,\delta}\|_{L^{\infty}}\iint_{\mathbb{T}^N\times\mathbb{R}^N}(1+|v|^q)|\nabla_{x}f^{n+1}||\nabla_{v}f^{n+1}||\nabla_{x} ~dvdx\cr
		&\quad +2\|\nabla_{x}T_{f^n}^{\epsilon,\delta}\|_{L^{\infty}}\iint_{\mathbb{T}^N\times\mathbb{R}^N}(1+|v|^q)| \nabla_{v}\nabla_{x}f^{n+1}| |\nabla_{v}f_\epsilon^n|~dvdx\cr
		&\leq C_{\epsilon,\delta,q}\left(  \|\nabla_{x}f^{n+1}\|^2_{L_{q}^2}+\|\nabla_{v}f^{n+1}\|^2_{L_q^2}
		+\gamma  \|\nabla_{x}\nabla_{v}f^{n+1}\|_{L_q^2}^2
		+C_{\gamma}\|\nabla_{v}f^{n+1}\|_{L_q^2}^2\right),
	\end{flalign*}
	where $\gamma$ will be chosen to be small enough so that the second order-derivative term can be absorbed into the last inequality of \eqref{I}.
	\noindent\newline
	$\bullet$ Estimate of $\rom{3}$: Using H\"older's inequality and Young's inequality, we get
	\begin{flalign*}
		\rom{3}
	&\leq 2\|\nabla_{x} u_{f^n}^{\epsilon}\|_{L^{\infty}}\iint_{\mathbb{T}^N\times\mathbb{R}^N}(1+|v|^q)|\nabla_{x}f^{n+1}||\nabla_{v}f^{n+1}|~dvdx \leq C_\e\|\nabla_{x}f^{n+1}\|_{L_q^2}^2+C_\e\|\nabla_{v}f^{n+1}\|_{L_q^2}^2.
	\end{flalign*}
	$\bullet$ Estimate of $\rom{4}$:
	One can easily see that
	\begin{flalign*}
		\rom{4}
		= 2N\|\nabla_{x}f^{n+1}\|_{L_q^2}^2.
	\end{flalign*}
	$\bullet$ Estimate of $ \rom{5} $:
Applying integration by parts, we get
	\begin{flalign*}
		 \rom{5} 
		&=-\iint_{\mathbb{T}^N\times\mathbb{R}^N}(1+|v|^q)(u_{f^n}^{\epsilon}-v)\cdot\nabla_{v}|\nabla_{x}f^{n+1}|^2~dvdx\cr
		&=-\iint_{\mathbb{T}^N\times\mathbb{R}^N}q|v|^{q}|\nabla_{x}f^{n+1}|^2~dvdx
		+\iint_{\mathbb{T}^N\times\mathbb{R}^N}q(u_{f^n}^{\epsilon}\cdot v)|v|^{q-2}|\nabla_{x}f^{n+1}|^2~dvdx\cr
		&\quad  -N\iint_{\mathbb{T}^N\times\mathbb{R}^N}(1+|v|^q)|\nabla_{x}f^{n+1}|^2~dvdx\cr
		&\leq q\iint_{\mathbb{T}^N\times\mathbb{R}^N}(1+|v|^q)|\nabla_{x}f^{n+1}|^2~dvdx
		+q\|u_{f^n}^{\epsilon}\|_{L^{\infty}}\iint_{\mathbb{T}^N\times\mathbb{R}^N}(1+|v|^q)|\nabla_{x}f^{n+1}|^2~dvdx\cr
		&\quad  +N\iint_{\mathbb{T}^N\times\mathbb{R}^N}(1+|v|^q)|\nabla_{x}f^{n+1}|^2~dvdx\cr
		&\leq C_{\epsilon,q,N}\|\nabla_{x}f^{n+1}\|_{L_q^2}^2.
	\end{flalign*}
	\noindent
	$\bullet$ Estimate of $\rom{6}$:  Note that
	 \begin{flalign*}
 \rom{6} 
&=-2\iint_{\mathbb{T}^N\times\mathbb{R}^N}qv|v|^{q-2}\nabla_{v}f^{n+1}: \nabla^2_{v}f^{n+1} T_{f^n}^{\epsilon,\delta}~dvdx-2\iint_{\mathbb{T}^N\times\mathbb{R}^N}(1+|v|^q)|\nabla_{v}^2f^{n+1}|^2T_{f^n}^{\epsilon,\delta}~dvdx
\end{flalign*}
where $  A:B:=\sum_{i=1}^m\sum_{j=1}^na_{ij}b_{ij}$ for $A,B\in \R^{m\times n}$.  
Since
\begin{align*}
v\nabla_{v}f: \nabla^2_{v}f&=\sum_{i=1}^N\sum_{j=1}^N v_i \pa_{v_j}f \pa_{v_i}\pa_{v_j}f=\frac 12  v  \cdot \nabla_v\{|\nabla_vf|^2\},
\end{align*}
we obtain
	 \begin{flalign*}
\rom{6}&=-q\iint_{\mathbb{T}^N\times\mathbb{R}^N}|v|^{q-2}v\cdot\nabla_v\{|\nabla_vf^{n+1}|^2\} T_{f^n}^{\epsilon,\delta}~dvdx-2\iint_{\mathbb{T}^N\times\mathbb{R}^N}(1+|v|^q)|\nabla_{v}^2f^{n+1}|^2T_{f^n}^{\epsilon,\delta}~dvdx \cr 
&\le q(N+q-2)\iint_{\mathbb{T}^N\times\mathbb{R}^N}|v|^{q-2}|\nabla_{v}f^{n+1}|^{2}T_{f^n}^{\epsilon,\delta}~dvdx\cr
&\leq  q(N+q-2)\|T_{f^n}^{\epsilon,\delta}\|_{L^{\infty}}\iint_{\mathbb{T}^N\times\mathbb{R}^N}(1+|v|^q)|\nabla_{v}f^{n+1}|^{2}~dvdx \cr
&\leq C_{\delta,q,N}\|\nabla_vf^{n+1}\|_{L_q^2}^2.
\end{flalign*}
	\noindent
$\bullet$ Estimate of $\rom{7}$ and $\rom{8}$:
It is straightforward that
	\begin{flalign*}
		\rom{7} + \rom{8} \le 2N\|\nabla_vf^{n+1}\|_{L_q^2}^2+\|\nabla_{v}f^{n+1}\|_{L_q^2}^2+\|\nabla_{x}f^{n+1}\|_{L_q^2}^2
	\end{flalign*}
	$\bullet$ Estimate of $ \rom{9} $:
	Similarly as in $ \rom{5} $, we obtain
	\begin{flalign*}
		 \rom{9} 
&=-2\iint_{\mathbb{T}^N\times\mathbb{R}^N}(1+|v|^q)\nabla_{v}f^{n+1}\cdot\Big[-\nabla_{v}f^{n+1}+\nabla_{v}^2f^{n+1}(u_{f^n}^{\epsilon}-v)\Big]~dvdx\cr
		&=2\iint_{\mathbb{T}^N\times\mathbb{R}^N}(1+|v|^q)|\nabla_{v}f^{n+1}|^2~dvdx
		+\iint_{{\mathbb{T}^N}\times\mathbb{R}^N}(q|v|^{q-2}v\cdot(u_{f^n}^{\epsilon}-v))|\nabla_{v}f^{n+1}|^{2}~dvdx\cr
		&\quad -N\iint_{\mathbb{T}^N\times\mathbb{R}^N}(1+|v|^q)|\nabla_{v}f^{n+1}|^{2}~dvdx\cr
		&\leq 2\iint_{\mathbb{T}^N\times\mathbb{R}^N}(1+|v|^q)|\nabla_{v}f^{n+1}|^2~dvdx +C_q\left(1+\|u_{f^n}^{\epsilon}\|_{L^{\infty}}\right)\iint_{\mathbb{T}^N\times\mathbb{R}^N}(1+|v|^q)|\nabla_{v}f^{n+1}|^{2}~dvdx
		\cr
		&\leq C_{\epsilon,q,N}\|\nabla_vf^{n+1}\|_{L_q^2}^2. 
	\end{flalign*}
Finally, combining all the estimates gives the desired result \eqref{derv}.

%
%
%
%
%
%

	\section*{Acknowledgments}
		Y.-P. Choi and  B.-H. Hwang were supported by National Research Foundation of Korea(NRF) grant funded by the Korean government(MSIP) (No. 2022R1A2C1002820).  Also, B.-H. Hwang was partially supported by Basic Science Research Program through the National Research Foundation of Korea(NRF) funded by the Ministry of Education(No. 2019R1A6A1A10073079).
	


	%
	%
	%
	%

\end{document}